\documentclass[11pt,english]{article}

\usepackage[numbers,comma]{natbib}
\usepackage[T1]{fontenc}
\usepackage[latin9]{inputenc}
\ifx\pdfoutput\undefined
\usepackage{graphicx}
\else
\usepackage[pdftex]{graphicx}
\fi
\usepackage{psfrag}    
\usepackage{subfigure} 
\usepackage{url}       
\usepackage{amsmath}
\usepackage{amsthm}
\usepackage{amsfonts}
\usepackage{appendix}
\usepackage{setspace}
\usepackage[top=1in, bottom=1in, left=1in, right=1in]{geometry} 
\usepackage{multirow}

\newtheorem{theorem}{Theorem}
\newtheorem{lemma}{Lemma}
\newtheorem{proposition}{Proposition}
\newtheorem{corollary}{Corollary}

\newtheorem{remark}{Remark}

\makeatletter
\usepackage{babel}
\makeatother

\begin{document}

\title{Critically loaded multi-server queues with abandonments, retrials, and time-varying parameters}

\author{Young Myoung Ko and Natarajan Gautam}
\maketitle
\begin{abstract}
In this paper, we consider modeling time-dependent multi-server queues that include abandonments and retrials. For the performance analysis of those, fluid and diffusion models called ``strong approximations'' have been widely used in the literature. Although they are proven to be asymptotically exact, their effectiveness as approximations in \emph{critically loaded regimes} needs to be investigated. To that end, we find that existing fluid and diffusion approximations might be either inaccurate under simplifying assumptions or computationally intractable. To address that concern, this paper focuses on developing a methodology by adjusting the fluid and diffusion models so that they significantly improve the estimation accuracy. We illustrate the accuracy of our adjusted models by performing a number of numerical experiments.
\end{abstract}

\section{Introduction} \label{sec_introduction}
In this paper, we are interested in the precise analysis of time-varying many-server queues with abandonments and retrials described in \citet{Mandelbaum:1998p1029} (See Figure \ref{fig_retrial}).
\begin{figure}
\centering
\includegraphics[width = .6\textwidth]{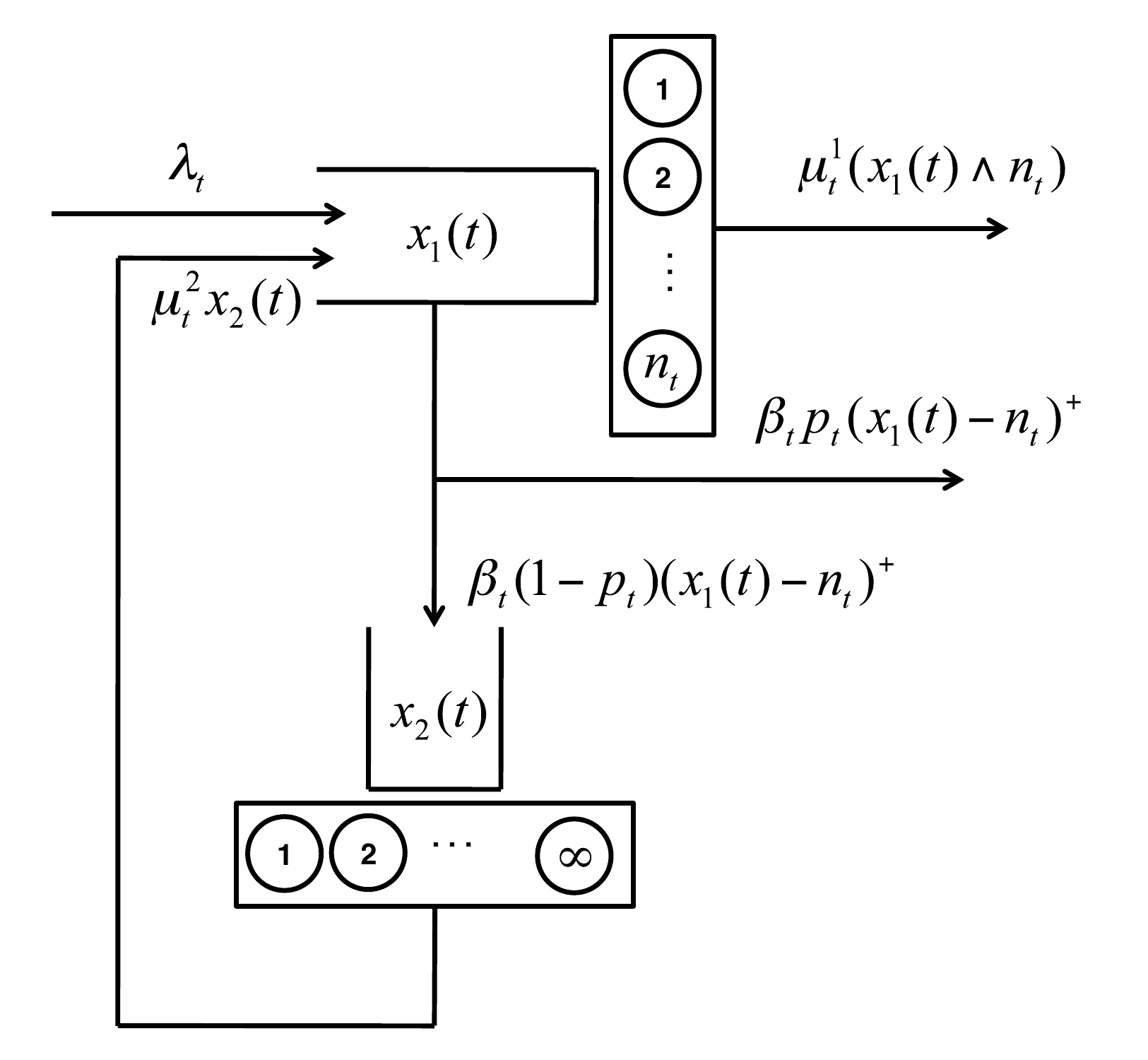}
\caption{Multi-server queue with abandonment and retrials, \citet{Mandelbaum:1998p1029}} \label{fig_retrial}
\end{figure}
Inspired by call centers, there have been extensive studies on multi-server queues, especially having a large number of servers. Most of the recent studies utilize asymptotic analysis as it makes the problem tractable and also provides good approximations under certain conditions. Asymptotic analysis, typically, utilizes weak convergence to fluid and diffusion limits which is nicely summarized in \citet{bill99} and \citet{Whitt02}. Methodologies to obtain fluid and diffusion limits, as described in \citet{Halfin81}, have been developed in the literature using two different ways in terms of the traffic intensity.\\
The first approach is to consider the convergence of a sequence of traffic intensities to a certain value. Relying on the value to which the sequence converges, there are three different operational regimes: efficiency driven (ED), quality and efficiency driven (QED), and quality driven (QD). Roughly speaking, if the traffic intensity ($\rho$) of the limit process is strictly greater than 1, it is called ED regime. If $\rho = 1$, then that is QED, otherwise QD. Many research studies have been done under the ED and QED regimes for multi-server queues like call centers (\citet{Halfin81}, \citet{Puhalskii:2000p1880}, \citet{Garnet:2002p1103}, \citet{Whitt04}, \citet{Whitt06b}, \citet{Pang:2009p1886}). Recently, the QED regime, also known as ``Halfin-Whitt regime'', has received a lot of attention; this is because it actually achieves both high utilization of servers 
and quality of service (\citet{Zeltyn:2005p1072}), and is a favorable operational regime for call centers with strict performance constraints (\citet{Mandelbaum:2009p1992}).\\
The second way to obtain limit processes is to accelerate parameters keeping the traffic intensity fixed. An effective methodology called ``uniform acceleration'' or ``strong approximations'' which enables the analysis of time-dependent queues (\citet{kurtz78}, \citet{mandelbaum95}, \citet{mandelbaum98}, \citet{massey98}, \citet{Whitt90}, \citet{Mandelbaum:1998p1029}, \citet{Hampshire:2009p1879}) is included in this scheme and in fact is the basis of this paper.\\
The advantage of the strong approximations as described in \citet{kurtz78} is that it can be applied to a wide class of stochastic processes and can be nicely extended to time-dependent systems by combining with the results in \citet{Mandelbaum:1998p1029}. However, it cannot be applied to multi-server queues directly due to an assumption that is not satisfied: i.e., for the diffusion model, the differentiability of the rate functions (e.g. net arrival rates and service rates) is necessary. But some rate functions are not differentiable everywhere since they are of the forms, $\min(\cdot,\cdot)$ or $\max(\cdot,\cdot)$. To extend the theory to non-smooth rate functions, \citet{Mandelbaum:1998p1029} proves weak convergence by introducing a new derivative called ``scalable Lipschitz derivative'' and provides models for several queueing systems such as Jackson networks, multi-server queues with abandonments and retrials, multi-class preemptive priority queues, etc. In addition, several sets of differential equations are also provided to obtain the mean value and covariance matrix of the limit processes. It, however, turns out that the resulting sets of differential equations are computationally intractable to solve in general and hence the theorems cannot be applied to obtain numerical values of performance measures. In a follow-on paper, \citet{Mandelbaum:2002p995} provides numerical results for queue lengths and waiting times in multi-server queues with abandonments and retrials by adding an assumption to deal with computational intractability. Specifically, the paper assumes measure zero at a set of time points where the fluid model hits non-differentiable points, which eventually enables us to apply Kurtz's diffusion models. However, as pointed out in \citet{Mandelbaum:2002p995}, if the system stays close to a critically loaded phase for a long time (i.e. \emph{lingering} around a non-differentiable point), their approach may cause significant inaccuracy.

To explain this inaccuracy in detail, consider a multi-server queue with abandonments and retrials as shown in Figure \ref{fig_retrial}. As an example we select numerical values $n_t=50, \mu_t^1=1,\mu_t^2=0.2$ for all $t$, whereas $\lambda_t$ alternates between $\lambda_t^1 = 45$ and $\lambda_t^2=55$ every two units of time (the parameters are defined in Section \ref{sec_problem} and illustrated in Figure \ref{fig_retrial}). Using the measure-zero assumption in \citet{Mandelbaum:2002p995}, we graph $E[x_1(t)]$ and $E[x_2(t)]$ in Figure \ref{fig_eg_inaccuracy} (a), and also $Var[x_1(t)]$, $Var[x_2(t)]$, and $Cov[x_1(t),x_2(t)]$ in Figure \ref{fig_eg_inaccuracy} (b). Notice that, although $E[x_1(t)]$ is reasonably accurate, the others ($E[x_2(t)]$, $Var[x_1(t)]$, $Var[x_2(t)]$, and $Cov[x_1(t),x_2(t)]$) are not accurate at all. The reason for that is the system lingers around the non-differentiable points. In addition, if one were to solve the differential equations numerically using computationally intractable techniques via the Lipschitz derivatives as described in \citet{Mandelbaum:1998p1029}, the similar level of inaccuracy occurs. We explain this in detail in Section \ref{subsec_inaccuracy}. However, this does nourish the need for a methodology to accurately predict the system performance which is the focus of this study.\\
\begin{figure}[htbp]
  \begin{center}
      \subfigure[ Simulation vs Fluid model]{
		\includegraphics[width = .8\textwidth]{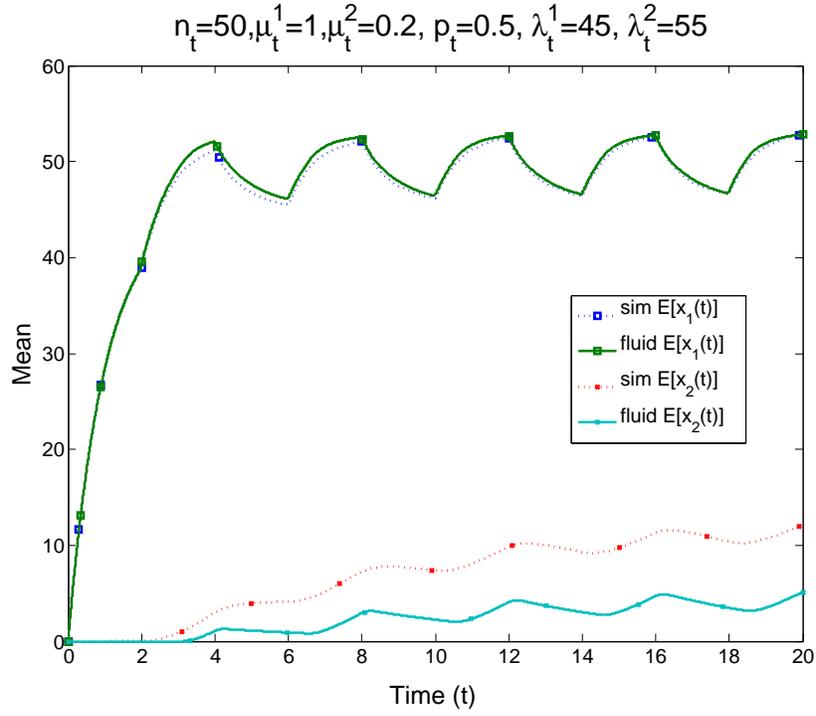}}
      \subfigure[Simulation vs Diffusion model]{
		\includegraphics[width = .8\textwidth]{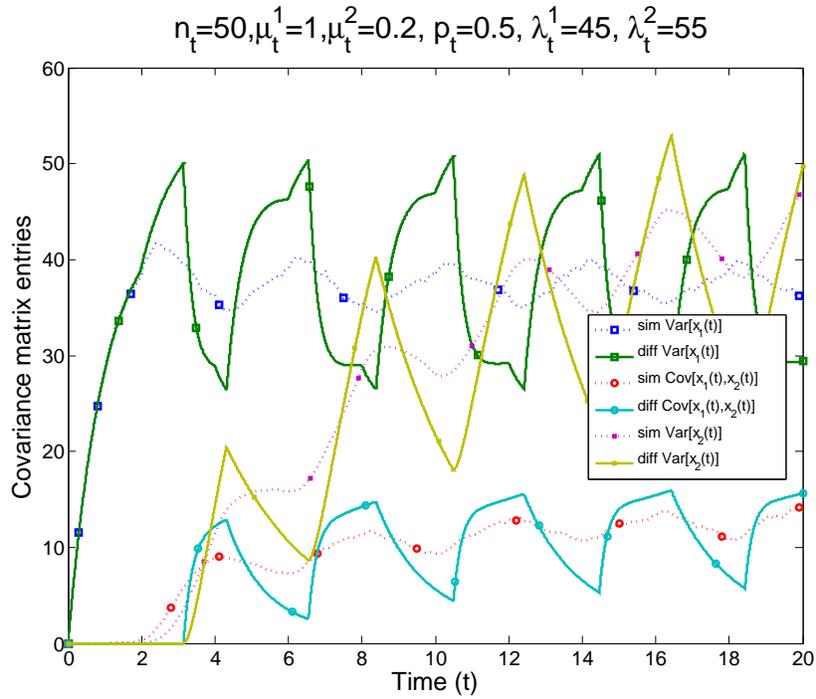}}
    \caption{Simulation vs Fluid and diffusion model with measure-zero assumption}
    \label{fig_eg_inaccuracy}
  \end{center}
\end{figure}
Having motivated the need to develop a methodology for the critically loaded phase, we now describe its importance. According to \citet{mandelbaum98} and \citet{Mandelbaum:2002p995}, time-dependent queues make transitions among three phases: underloaded, critically loaded, and overloaded. The phase of the system is determined by the fluid model. The limit process in the strong approximations does not require any regimes such as QD, QED, or ED. However, from Section 1.4 in \citet{Zeltyn:2005p1072}, we could find a rough correspondence between the operational regimes (QD, QED, and ED) and the phases in time-varying queues (underloaded, critically loaded, and overloaded). Recall that the QED regime is favorable to the operation of the call centers. Therefore, capturing the dynamics of multi-server queues in the critically loaded phase is also of significant importance. Nonetheless, from Figure \ref{fig_eg_inaccuracy}, we found two major issues in the existing approach: 1) the fluid model (where the non-differentiability issue is actually irrelevant) is itself inaccurate and 2) sharp spikes which cause massive estimation errors are observed at the non-differentiable points in the diffusion model in contrast to the smooth curves in the simulation. In this paper, we approach the above two issues from a different point of view and provide an effective solution to them. Considering those, the contributions of this paper can be summarized as follows:
\begin{enumerate}
\item To the best of our knowledge, inaccuracy in the fluid model has never been addressed in the literature. We explain why it happens and ameliorate the fluid model.
\item Sharp spikes observed in the diffusion model cannot be resolved using the methodology in the literature. We provide a reasonable approximation-methodology so that it could smoothen the spikes and improve the estimation accuracy dramatically.
\end{enumerate}
We now describe the organization of this paper. In Section \ref{sec_problem}, we state the problem considered in this paper. In Section \ref{sec_strong}, we summarize the strong approximations in \citet{kurtz78} and \citet{Mandelbaum:1998p1029}, and describe the above issues in detail. In Section \ref{sec_adjustedfluid}, we construct an adjusted fluid model to estimate the exact mean value of the system state. However, this would not immediately result in a computationally feasible approach. For that, in Section \ref{sec_adjusteddiffusion}, we explain our Gaussian-based approximations to achieve computational feasibility and smoothness in the diffusion model. Further investigation on the adjusted models is provided in Section \ref{sec_g} to show how actually our adjusted models contribute to the estimation accuracy. In Section \ref{sec_numerical}, we provide a number of numerical examples and compare against the existing approach as well as simulation. Finally, in Section \ref{sec_conclusion}, we make concluding remarks and explain directions for future work.
\section{Problem description} \label{sec_problem}
Consider Figure \ref{fig_retrial} that illustrates a multi-server queue with abandonments and retrials as described in \citet{Mandelbaum:1998p1029} and \citet{Mandelbaum:2002p995}. There are $n_t$ number of servers in the service node at time $t$. Customers arrive to the service node according to a non-homogeneous Poisson process at rate $\lambda_t$. The service time of each customer follows a distribution having a memoryless property at rate $\mu_t^1$. Customers in the queue are served under the FCFS policy and the abandonment rate of customers is $\beta_t$ with exponentially distributed time to abandon. Abandoning customers leave the system with probability $p_t$ or go to a retrial queue with probability $1-p_t$. The retrial queue is equivalent to an infinite-server-queue and hence each customer in the retrial queue waits there for a random amount of time with mean $1/\mu_t^2$ and returns to the service node.\\
Let $X(t) = \big(x_1(t),x_2(t)\big)$ be the system state where $x_1(t)$ is the number of customers in the service node and $x_2(t)$ is the number of customers in the retrial queue. Then, $X(t)$ is the unique solution to the following integral equations:
\begin{eqnarray}
  x_1(t) &=& x_1(0) +Y_1\Big(\int_{0}^{t}\lambda_s ds\Big) + Y_2\Big(\int_{0}^{t}x_2(s)\mu_s^2ds\Big) - Y_3\Big(\int_{0}^{t}\big(x_1(s)\wedge n_s\big)\mu_s^1ds\Big) \nonumber \\
  && - Y_4\Big(\int_{0}^{t}\big(x_1(s)-n_s\big)^+\beta_s(1-p_s)ds\Big) - Y_5\Big(\int_{0}^{t}\big(x_1(s)-n_s\big)^+\beta_sp_sds\Big), \label{eqn_rx1} \\
  x_2(t) &=& x_2(0) + Y_4\Big(\int_{0}^{t}\big(x_1(s)-n_s\big)^+\beta_s(1-p_s)ds\Big) - Y_2\Big(\int_{0}^{t}x_2(s)\mu_s^2ds\Big), \label{eqn_rx2}
\end{eqnarray}
where $Y_i$'s are independent rate-$1$ Poisson processes.\\
The performance measures we are interested in are $E[X(t)]$ and $Cov[X(t),X(t)]$ (i.e. $Var[x_1(t)]$, $Var[x_2(t)]$, and $Cov[x_1(t),x_2(t)]$) for any given time $t \in [0,T]$, where $T < \infty$ is a constant. Especially, we have an interest in the system that is \emph{lingering} near the critically loaded phase for a long time. Anyhow, as one may notice, the above two equations (\ref{eqn_rx1}) and (\ref{eqn_rx2}) cannot be solved directly. If all the parameters are constant, i.e. $\lambda_t = \lambda, \mu_t^1 = \mu^1, \mu_t^2=\mu^2, \beta_t =\beta,$ and $p_t = p$, one can consider a Continuous Time Markov Chain (CTMC) model to obtain the performance measures. However, even assuming constant parameters, calculating the performance measures at any given time $t$ is hard since $x_1(t)$ and $x_2(t)$ both are unbounded and solving balance equations in the two or more dimensional spaces requires tremendous efforts. Furthermore, when the number of servers is large, computational issues might arise. We, accordingly, would try to take advantage of an asymptotic methodology that is adequate for the analysis of time-varying systems with large number of servers. Nevertheless, as briefly mentioned in Section \ref{sec_introduction}, we found that the existing methodologies are either computationally intractable or significantly inaccurate in the critically loaded phase. The objective of this paper is to develop a new approach to enhance the accuracy in estimating the mean value and covariance matrix for the multi-server queues with abandonments and retrials.\\
To do so, we start by summarizing the strong approximations and addressing the potential limitations in the following section.  
\section{Summary of the strong approximations} \label{sec_strong}
In Section \ref{subsec_strong}, we recapitulate the strong approximations in \citet{kurtz78} and \citet{Mandelbaum:1998p1029}. In Section \ref{subsec_inaccuracy}, we explain what produces estimation errors and why existing methodologies do not fix them.  
\subsection{Strong approximations} \label{subsec_strong}
In this section, we review the fluid and diffusion approximations developed by \citet{kurtz78} that we would leverage upon for our methodology. We also briefly mention the result in \citet{Mandelbaum:1998p1029} which extends Kurtz's result to models involving non-smooth rate functions. Moreover, it is worthwhile to note that for $n\in \mathbf{N}$, the state of the queueing system $X_n(t)$ includes jumps but the limit process is continuous. Therefore, the weak convergence result that is presented is with respect to uniform topology in Space $D$ (\citet{bill99} and \citet{Whitt02}).\\
Let $X(t)$ be an arbitrary $d$-dimensional stochastic process which is the solution to the following integral equation:
\begin{eqnarray}
	X(t) = x_0 + \sum_{i=1}^{k} l_i Y_{i} \bigg(\int_{0}^{t} f_{i}\big(s,X(s)\big)ds \bigg), \label{eqn_001}
\end{eqnarray}
where $x_0 = X(0)$ is a constant, $Y_{i}$'s are independent rate-$1$ Poisson processes, $l_i \in \mathbf{Z}^d$ for $i \in \{1,2,\ldots, k\}$ are constant, and $f_i$'s are continuous functions such that $|f_i(t,x)| \le C_i (1+|x|)$ for some $C_i < \infty$, $t\le T$ and $T<\infty$. Note that we just consider a finite number of $l_i$'s to simplify proofs, which is reasonable for real world applications. \\
Notice that a special case of $X(t)$ in equation (\ref{eqn_001}) is the $X(t)$ we described in equations (\ref{eqn_rx1}) and (\ref{eqn_rx2}) in our problem explained in Section \ref{sec_problem}. Following the notation in equation (\ref{eqn_001}), we have, for our problem in Section \ref{sec_problem}, $x=(x_1,x_2)$ and $t\le T$, 
\begin{align*}
  f_1(t,x) &= \lambda_t,  f_2(t,x) = \mu_t^2 x_2,   f_3(t,x) = \mu_t^1(x_1 \wedge n_t),\\
  f_4(t,x) &= \beta_t(1-p_t)(x_1-n_t)^+,  f_5(t,x) = \beta_tp_t(x_1-n_t)^+, \\
 l_1 &= \binom{1}{0}, l_2 = \binom{1}{-1},  l_3 = \binom{-1}{0}, l_4=\binom{-1}{1}, \textrm{ and } l_5=\binom{-1}{0}.
\end{align*}
Coming back to the generalized $X(t)$ process, we reiterate that it is usually not tractable to solve the integral equation (\ref{eqn_001}). Therefore, to approximate the $X(t)$ process, define a sequence of stochastic processes $\{X_n(t)\}$ which satisfy the following integral equation:
\begin{eqnarray*}
	X_n(t) = x_0 + \sum_{i=1}^{k} \frac{1}{n} l_i Y_{i} \bigg(\int_{0}^{t} n f_{i}\big(s, X_n(s)\big)ds \bigg). \label{eqn_002}
\end{eqnarray*}
Typically the process $X_n(t)$ (usually called a scaled process) is obtained by taking $n$ times faster rates of events and $1/n$ of the increment of the system state. This type of setting is used in the literature and is denoted as ``uniform acceleration'' in \citet{massey98}, \citet{Mandelbaum:1998p1029}, and \citet{Mandelbaum:2002p995}. Then, the following theorem provides the fluid model to which $\{X_n(t)\}$ converges almost surely as $n\rightarrow \infty$. Define
\begin{eqnarray}
	F(t,x) = \sum_{i=1}^{k} l_i f_{i}(t,x) \label{eqn_F}.
\end{eqnarray}
\begin{theorem}[Fluid model, \citet{kurtz78}] \label{theo_fluid}
If there is a constant $M < \infty$ such that $|F(t,x)-F(t,y)| \le M|x-y|$ for all $t \le T$ and $T<\infty$. Then, $\lim_{n \rightarrow \infty} X_n(t) = \bar{X}(t)$ a.s. where $\bar{X}(t)$ is the solution to the following integral equation:
\begin{eqnarray*}
	\bar{X}(t) = x_0 + \sum_{i=1}^{k} l_i \int_{0}^{t} f_{i}\big(s, \bar{X}(s)\big)ds. \label{eqn_003}
\end{eqnarray*}
\end{theorem}
Note that $\bar{X}(t)$ is a deterministic time-varying quantity. We will subsequently connect $\bar{X}(t)$ and $X(t)$ defined in equation (\ref{eqn_001}), but before that we provide the following result. Once we have the fluid model, we can obtain the diffusion model from the scaled centered process ($D_n(t)$). Define $D_n(t)$ to be $\sqrt{n}\big(X_n(t) - \bar{X}(t)\big)$. Then, the limit process of $D_n(t)$ is provided by the following theorem.
\begin{theorem}[Diffusion model, \citet{kurtz78}] \label{theo_diffusion}
If $f_i$'s and $F$, for some $M<\infty$, satisfy
\begin{eqnarray*}
	|f_i(t,x)-f_i(t,y)| \le M|x-y| \quad \textrm{and} \quad  \bigg| \frac{\partial}{\partial x_i}F(t,x)\bigg| \le M, \qquad \textrm{for } i \in \{1,\ldots,k\} \textrm{ and } 0\le t\le T,
\end{eqnarray*}
then $\lim_{n \rightarrow \infty} D_n(t) = D(t)$ where $D(t)$ is the solution to
\begin{eqnarray*}
	D(t) = \sum_{i=1}^{k} l_i \int_{0}^{t} \sqrt{f_i\big(s,\bar{X}(s)\big)}dW_i(s) + \int_{0}^{t} \partial F\big(s,\bar{X}(s)\big)D(s) ds, \label{eqn_004}
\end{eqnarray*}
$W_i(\cdot)$'s are independent standard Brownian motions, and $\partial F(t,x)$ is the gradient matrix of $F(t,x)$ with respect to $x$.
\end{theorem}
\begin{remark} \label{rem_nondiff}
Theorem \ref{theo_diffusion} requires that $F(\cdot,\cdot)$ has a continuous gradient matrix. Therefore, if we don't have such an $F$, then we cannot apply Theorem \ref{theo_diffusion} directly to obtain the diffusion model.
\end{remark}
\begin{remark} \label{rem_gaussian}
According to \citet{ethier86}, if $D(0)$ is a constant or a Gaussian random vector, then $D(t)$ is a Gaussian process.
\end{remark}
Now, we have the fluid and diffusion models for $X_n(t)$. Therefore, for a large $n$, $X_n(t)$ is approximated by
\begin{eqnarray*}
	X_n(t) \approx \bar{X}(t) + \frac{D(t)}{\sqrt{n}}. \label{eqn_005}
\end{eqnarray*}
If we follow this approximation, we can also approximate the mean and covariance matrix of $X_n(t)$ denoted by $E\big[X_n(t)\big]$ and $Cov \big[X_n(t),X_n(t)\big]$ respectively as
\begin{eqnarray}
	E\big[X_n(t)\big] &\approx& \bar{X}(t) + \frac{E\big[D(t)\big]}{\sqrt{n}}, \label{eqn_006} \\
	Cov \big[X_n(t),X_n(t)\big] &\approx& \frac{Cov \big[D(t),D(t)\big]}{n}. \label{eqn_007}
\end{eqnarray}
In equations (\ref{eqn_006}) and (\ref{eqn_007}), only $\bar{X}(t)$ is known. Therefore, in order to get approximated values of $E\big[X_n(t)\big]$ and $Cov \big[X_n(t),X_n(t)\big]$, we need to obtain $E\big[D(t)\big]$ and $Cov\big[D(t),D(t)\big]$. The following theorem provides a methodology to obtain $E\big[D(t)\big]$ and $Cov \big[D(t),D(t)\big]$.
\begin{theorem}[Mean and covariance matrix of linear stochastic systems, \citet{arnold92}] \label{theo_moment}
Let $Y(t)$ be the solution to the following linear stochastic differential equation.
\begin{eqnarray*}
	dY(t) = A(t)Y(t)dt + B(t)dW(t), \quad Y(0)=0, \label{eqn_008}
\end{eqnarray*}
where $A(t)$ is a $d \times d$ matrix, $B(t)$ is a $d \times k$ matrix, and W(t) is a $k$-dimensional standard Brownian motion.
Let $M(t) = E\big[Y(t)\big]$ and $\Sigma(t) = Cov\big[Y(t), Y(t)\big]$. Then, $M(t)$ and $\Sigma(t)$ are the solution to the following ordinary differential equations:
\begin{eqnarray}
	\frac{d}{dt}M(t) &=& A(t) M(t) \label{eqn_009} \nonumber \\
	\frac{d}{dt}\Sigma(t) &=& A(t) \Sigma(t) + \Sigma(t) A(t)' + B(t)B(t)'. \label{eqn_010}
\end{eqnarray}
\end{theorem}
\begin{corollary} \label{cor_moment}
If $M(0)=0$, then $E\big[M(t)\big] = 0$ for $t \ge 0$.
\end{corollary}
By Corollary \ref{cor_moment}, if $D(0)=0$, then $E\big[D(t)\big] = 0$ for $t \ge 0$. Therefore, if $\bar{X}(0) = X(0) = x_0$, then we can rewrite (\ref{eqn_006}) to be
\begin{eqnarray*}
	E\big[X_n(t)\big] &\approx& \bar{X}(t). \label{eqn_011}
\end{eqnarray*}
Recalling Remark \ref{rem_nondiff}, the diffusion model in \citet{kurtz78} requires differentiability of rate functions. Otherwise, we cannot apply Theorem \ref{theo_diffusion}. To get this problem under control, \citet{Mandelbaum:1998p1029} introduces a new derivative called ``scalable Lipschitz derivative'' and proves weak convergence using it. Unlike the result in \citet{kurtz78}, it turns out that the diffusion limit may not be a Gaussian process when rate functions are not differentiable everywhere. In \citet{Mandelbaum:1998p1029}, expected values of the diffusion model may not be zero (compare it with Corollary \ref{cor_moment}) and could adjust the inaccuracy in the fluid model (see \citet{Mandelbaum:2002p995}). The resulting differential equations for the diffusion model, however, are computationally intractable. For example, in \citet{Mandelbaum:1998p1029}, one of the differential equations has the following form:
\begin{eqnarray}
	\frac{d}{dt}E\big[Q_1^{(1)}(t)\big] &=& (\mu_t^1 \mathbf{1}_{\{Q_1^{(0)} \le n_t\}} + \beta_t \mathbf{1}_{\{Q_1^{(0)} > n_t\}})E\big[Q_1^{(1)}(t)^-\big] \nonumber \\ 
		&& - (\mu_t^1 \mathbf{1}_{\{Q_1^{(0)} < n_t\}} + \beta_t \mathbf{1}_{\{Q_1^{(0)} \ge n_t\}})E\big[Q_1^{(1)}(t)^+\big] + \mu_t^2E\big[Q_2^{(1)}(t)\big], \label{eqn_actdiff}
\end{eqnarray}
rendering it to be intractable.\\
Therefore, \citet{Mandelbaum:2002p995}, as we understand, resorts to the method in \citet{kurtz78} by assuming measure zero at non-smooth points to avoid computational difficulty.\\
\subsection{Inaccuracy of strong approximations} \label{subsec_inaccuracy}
Though not mentioned in any previous studies, to the best of our knowledge, the fluid model has the possibility of being inaccurate when approximating the mean value of the system state. Consider the actual integral equation to get the exact value of $E\big[X(t)\big]$ by the following theorem.
\begin{theorem}[Expected value of $X(t)$] \label{theo_exp}
Consider $X(t)$ defined in equation (\ref{eqn_001}). Then, for $t\le T$, $E\big[X(t)\big]$ is the solution to the following integral equation.
\begin{eqnarray}
	E\big[X(t)\big] = x_0 + \sum_{i=1}^k l_i \int_{0}^{t}E\Big[f_i\big(s,X(s)\big)\Big] ds \label{eqn_013}
\end{eqnarray}
\begin{proof}
Take expectation on both sides of equation (\ref{eqn_001}). Then, 
\begin{eqnarray}
	E\big[X(t)\big] &=& x_0 + \sum_{i=1}^k l_i E\Bigg[Y_i\bigg(\int_{0}^{t}f_i\big(s,X(s)\big)ds \bigg)\Bigg] \nonumber \\
		&=& x_0 + \sum_{i=1}^k l_i E\bigg[\int_{0}^{t}f_i\big(s, X(s)\big)ds \bigg] \textrm{ since $Y_i(\cdot)$'s are non-homogeneous Poisson processes} \nonumber \\
		&=& x_0 + \sum_{i=1}^k l_i \int_{0}^{t}E\Big[f_i\big(s, X(s)\big)\Big] ds \textrm{  by Fubini theorem in \citet{folland99}.} \nonumber
\end{eqnarray}
Therefore, we prove the theorem. 
\end{proof}
\end{theorem}
Comparing Theorems \ref{theo_fluid} and \ref{theo_exp}, notice that we cannot conclude that $\bar{X}(t)$ in Theorem \ref{theo_fluid} and $E\big[X(t)\big]$ in Theorem \ref{theo_exp} are close enough since $E\big[f_i(t, X(t))\big] \neq f_i\big(t, E[X(t)]\big)$. In some applications, $f_i$'s might be constants or linear combinations of components of $X(t)$. In those cases, Theorem \ref{theo_exp} and the following corollary imply that the fluid model would be the exact mean value of the system state.
\begin{corollary} \label{cor_exp}
If $f_i(t,x)$'s are constants or linear combinations of the components of $x$, Then,
\begin{eqnarray}
	E[X(t)] = \bar{X}(t), \nonumber
\end{eqnarray}
where $X(t)$ is the solution to (\ref{eqn_001}) and $\bar{X}(t)$ is the deterministic fluid model from theorem \ref{theo_fluid}.
\begin{proof}
Using linearity of expectation in \citet{williams91}, we can obtain the same integral equation for both $E\big[X(t)\big]$ and $\bar{X}(t)$.
\end{proof}
\end{corollary}
However, if we have different forms of $f_i$'s where $E\big[f_i(t, X(t))\big] \neq f_i\big(t, E[X(t)]\big)$, then the fluid model would be inaccurate. Notice that the fluid model does not require the differentiability of rate functions in both \citet{kurtz78} and \citet{Mandelbaum:1998p1029}. Therefore, in this problem, the differentiability issue in rate functions is actually irrelevant.\\ 
\begin{figure}[htbp]
  \begin{center}
      \subfigure[ Simulation vs Fluid model]{
		\includegraphics[width = .8\textwidth]{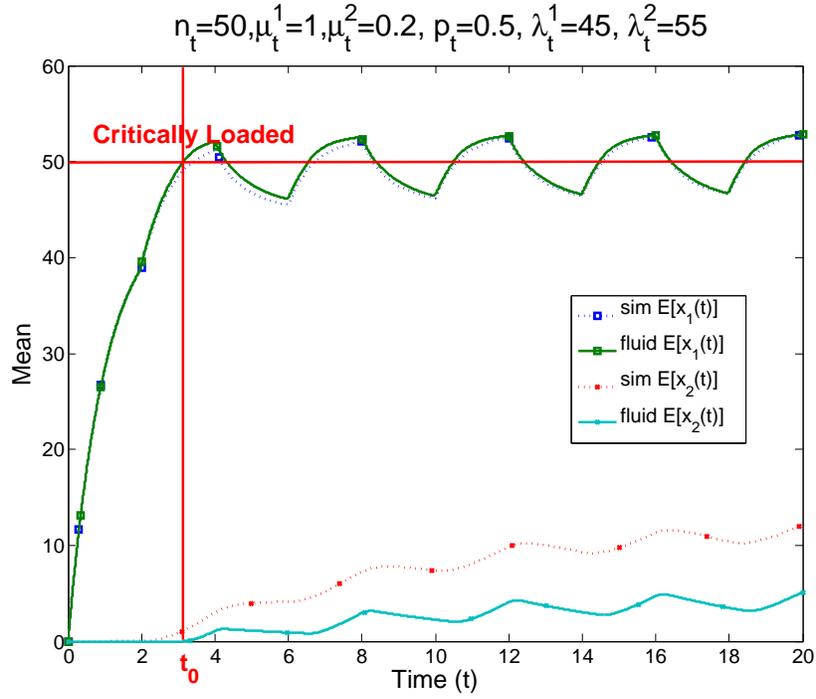}}
      \subfigure[Simulation vs Diffusion model]{
		\includegraphics[width = .8\textwidth]{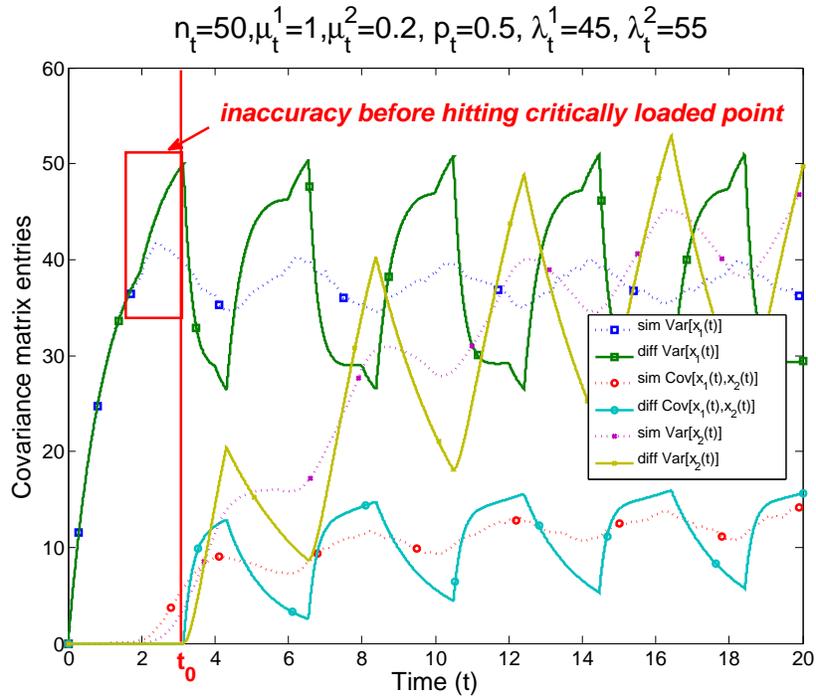}}
    \caption{Simulation vs Fluid and diffusion model with measure-zero assumption}
    \label{fig_eg_annotated}
  \end{center}
\end{figure}
Now we move our attention to inaccuracy in the diffusion model. We use the annotated version of Figure \ref{fig_eg_inaccuracy} (via Figure \ref{fig_eg_annotated}) here for the clear explanation. Figures \ref{fig_eg_annotated} (a) and (b) show the mean value and covariance matrix of the system against those of the simulation respectively. Since the number of servers is $50$, as shown in Figure \ref{fig_eg_annotated} (a), the mean value of $x_1(t)$ is fluctuating near the critically loaded point. From the figure, we also confirm that the fluid model is quite inaccurate for the mean value of $x_2(t)$. For the covariance matrix, as shown in Figure \ref{fig_eg_annotated} (b), the diffusion model brings about immense estimation errors (sharp spikes) in the vicinity of the critically loaded time points. Notice that from Figure \ref{fig_eg_annotated} (b) we found that \emph{even if the differential equations such as equation (\ref{eqn_actdiff}) in \citet{Mandelbaum:1998p1029}, which are known to be true, can be numerically solvable, it does not contribute to improving the estimation accuracy}. In the figure, the time point $t_0$ is the time when the fluid model hits a critically loaded point for the first time. \emph{The differential equations in \citet{Mandelbaum:1998p1029} are virtually same as those in \citet{Mandelbaum:2002p995} which assume measure zero for the computational tractability until the fluid model reaches a critically loaded point for the first time}. Therefore, we can think that the graphs before time $t_0$ in Figure \ref{fig_eg_annotated} are exactly same as those obtained from the methodology in \citet{Mandelbaum:1998p1029} though we could not get the graphs after $t_0$. However, as seen in Figure \ref{fig_eg_annotated} (b), the estimation errors become apparent much earlier than the time point $t_0$. Therefore, we figure out that the methodology in \citet{Mandelbaum:1998p1029} does not remove the sharp spikes at least until the time $t_0$. Moreover, from the shapes of the differential equations, we would conjecture that the methodology in \citet{Mandelbaum:1998p1029} might not get rid of the sharp spikes even after the time $t_0$. The drift matrix of the diffusion model in \citet{Mandelbaum:1998p1029} still makes sudden changes at the critically loaded point which actually causes the spikes. We will revisit and explain it in Section \ref{sec_g}.\\
In the next two sections, we describe our approach to the above issues in both fluid and diffusion models. In Section \ref{sec_adjustedfluid}, we address the inaccuracy in the fluid model by a constructing new process. In particular, in Section \ref{sec_adjusteddiffusion}, based on the adjusted fluid model, we explain how to remove the sharp spikes that causes vast estimation errors in the diffusion model. 
\section{Adjusted fluid model} \label{sec_adjustedfluid}
The basic idea of our approach is to construct a new process, $Z(t)$), so that its fluid model is exactly the same as the mean value of the original process $X(t)$ as described in Theorem \ref{theo_exp} (this is schematically explained in Figure \ref{fig_newprocess}).
\begin{figure}
\centering
\includegraphics[width = .5\textwidth]{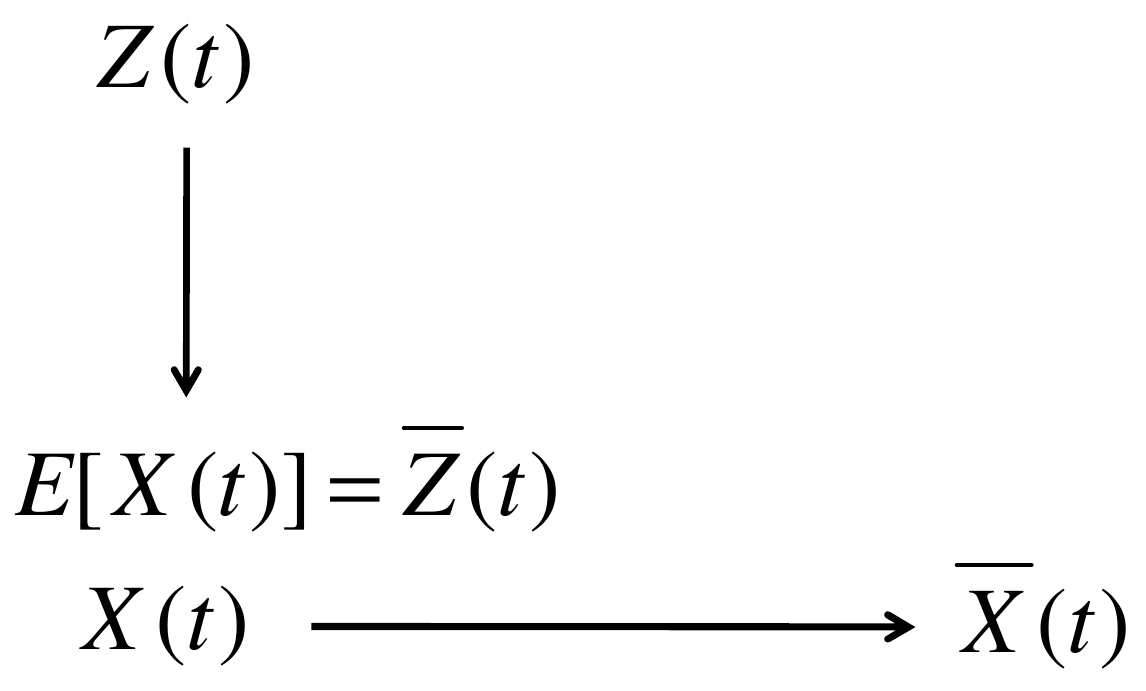}
\caption{Construction of a new process} \label{fig_newprocess}
\end{figure}
Although we concentrate on multi-server queues, this approach can be applied to more general types of stochastic systems. Therefore, we borrow the more general notation in Section \ref{subsec_strong} (as opposed to that in Section \ref{sec_problem}).\\
To begin with, define a set $\mathbb{F}$ of all distribution functions that have a finite mean and covariance matrix in $\mathbf{R}^d$. This set is valid for the fluid model since conditions on $f_i$'s guarantee that $E\big[|X(t)|\big] < \infty$ and $|Cov[X(t),X(t)]| < \infty$ for all $t \le T$. Define a subset $\mathbb{F}_0$ of $\mathbb{F}$  such that any $h \in \mathbb{F}_0$ has zero mean. We call an element of $\mathbb{F}_0$ a ``base distribution'' for the remainder of this paper.
\begin{proposition} \label{prop_exp}
For $t \le T$ and $i \in {1,2, \ldots, k}$, let $\mu(t) = E[X(t)]$. Then, $E\big[f_i(t,X(t))\big]$ can be represented as a function of $\mu(t)$, i.e., there exists a function $g_i(t,\cdot)$ such that
\[g(t,\mu(t)) = E\big[f_i(t,X(t))\big].\]
\begin{proof}
For fixed $t_0 \le T$, suppose the distribution of $X(t_0)$ is $F$. Then, $F \in \mathbb{F}$. For $F \in \mathbb{F}$, we can always find $F_0 \in \mathbb{F}_0$ such that $F(x) = F_0(x-\mu)$ where $\mu = E[X(t_0)] = \int_{\mathbf{R}^d} x dF$. Then,
\begin{eqnarray*}
	E\big[f_i(t_0,X(t_0))\big] &=& \int_{\mathbf{R}^d} f_i(t_0,x) dF \\
		&=& \int_{\mathbf{R}^d} f_i(t_0,x+\mu) dF_0.
\end{eqnarray*}
Since the integration removes $x$, by making $t_0$ and $\mu$ variables (i.e. substitute $t_0$ and $\mu$ with $t$ and $\mu(t)$ respectively), we have
\begin{eqnarray*}
	E\big[f_i(t,X(t))\big] = g_i(t,\mu(t)), \textrm{ for some function } g_i.
\end{eqnarray*}
\end{proof}
\end{proposition}
\begin{remark}
Proposition \ref{prop_exp} does not mean that $\mu(\cdot)$ completely identifies the function $g_i(\cdot,\cdot)$. In fact, the function $g_i(\cdot,\cdot)$ might be unknown unless the base distribution is identified but we can say that such a function $g_i(\cdot,\cdot)$ exists.
\end{remark}
For $t  \le T$, let $\mu(t) = E\big[X(t)\big]$. Let $g_i\big(t,\mu(t)\big) = E\big[f_i(t,X(t))\big]$ for $i \in \{1, \ldots, k\}$. Then, we can construct a new stochastic process $Z(t)$ which is the solution to the following integral equation:
\begin{eqnarray}
	Z(t) = z_0 + \sum_{i=1}^{k} l_i Y_{i} \bigg(\int_{0}^{t} g_{i}\big(s,Z(s)\big)ds \bigg). \label{eqn_014}
\end{eqnarray}
Based on equation (\ref{eqn_014}), define a sequence of stochastic processes $\{Z_n(t)\}$ satisfying
\begin{eqnarray}
	Z_n(t) = x_0 + \sum_{i=1}^{k} \frac{1}{n} l_i Y_{i} \bigg(\int_{0}^{t} n g_{i}\big(s, Z_n(s)\big)ds \bigg). \label{eqn_adjseq}
\end{eqnarray}
Next, we would like to obtain the fluid model for $Z_n(t)$. Before doing that, we need to check whether the functions $g_i$'s satisfy the conditions to apply Theorem \ref{theo_fluid}. Following lemmas show that $g_i$'s actually meet those conditions. The proofs of the lemmas are provided in Appendix \ref{app_lemmas}.
\begin{lemma} \label{lem_condition1}
If $|f_i(t,x)| \le C_i (1+|x|)$ for $t\le T$, then $g_i(t,x)$'s satisfy
\begin{eqnarray*}
	|g_i(t,x)| &\le& D_i (1+|x|) \quad \textrm{for some } D_i < \infty. \label{eqn_015} 
\end{eqnarray*}
\end{lemma}
For the next lemma, we would like to define
\begin{eqnarray}
  G(t,x) = \sum_{i=1}^k l_i g_i(t,x). \label{eqn_G}
\end{eqnarray}
\begin{lemma} \label{lem_condition2}
For $t\le T$, if $|f_i(t,x)-f_i(t,y)| \le M |x-y|$, then $g_i(t,x)$'s satisfy
\begin{eqnarray*}
	|g_i(t,x) - g_i(t,y)| \le M|x-y|,
\end{eqnarray*}
and if $|F(t,x)-F(t,y)| \le M |x-y|$, then $G(t,x)$ satisfies
\begin{eqnarray*}
	|G(t,x) - G(t,y)| \le M|x-y|. \label{eqn_019} 
\end{eqnarray*}
\end{lemma}
Lemmas \ref{lem_condition1} and \ref{lem_condition2} show that if $f_i$'s satisfy the conditions to obtain the fluid limit of $X_n(t)$, then $g_i$'s are also eligible for the fluid model of $Z_n(t)$. Therefore, we are now able to provide the adjusted fluid model based on Lemmas \ref{lem_condition1} and \ref{lem_condition2}.
\begin{theorem}[Adjusted fluid model] \label{theo_modfluid}
Assume
\begin{eqnarray}
	\big|f_i(t,x)\big| &\le& C_i\big(1+|x|\big) \quad \textrm{for } i\in \{1,\ldots, k\}, \label{eqn_024}\\
	\big|F(t,x)-F(t,y)\big| &\le& M|x-y|. \label{eqn_025}
\end{eqnarray}
Then, $\lim_{n \rightarrow \infty} Z_n(t) = \bar{Z}(t)$ a.s., where $\bar{Z}(t)$ is the solution to the following integral equation:
\begin{eqnarray}
	\bar{Z}(t) = x_0 + \sum_{i=1}^{k} l_i \int_{0}^{t} g_{i}\big(s, \bar{Z}(s)\big)ds, \label{eqn_026}
\end{eqnarray}
and furthermore
\begin{eqnarray}
	\bar{Z}(t) = E\big[X(t)\big] = x_0 + \sum_{i=1}^k l_i \int_{0}^{t}E\Big[f_i\big(s,X(s)\big)\Big] ds. \label{eqn_027}
\end{eqnarray}
\begin{proof}
From Lemmas \ref{lem_condition1} and \ref{lem_condition2}, (\ref{eqn_024}) and (\ref{eqn_025}) imply
\begin{eqnarray*}
	|g_i(t,x)|\le D_i (1+|x|) \quad \textrm{and} \quad |G(t,x) - G(t,y)| \le M|x-y|. \label{eqn_029} 
\end{eqnarray*}
Therefore, by Theorem \ref{theo_fluid}, we have equation (\ref{eqn_026}), and by the definition of $g_i(t,x)$'s, we have equation (\ref{eqn_027}).
\end{proof}
\end{theorem}
Comparing equation (\ref{eqn_027}) with equation (\ref{eqn_013}) in Theorem \ref{theo_exp}, we notice that Theorem \ref{theo_modfluid} via equation (\ref{eqn_027}) could provide the exact estimation of $E\big[X(t)\big]$. Though Theorem \ref{theo_modfluid} provides the exact estimation of $E\big[X(t)\big]$, the functions $g_i$'s cannot be identified unless the base distribution is known, which forces us to develop an algorithm to find $g_i$'s. Nonetheless, when applying our adjusted fluid model to the multi-server queues with abandonments and retrials, we, in fact, have a good candidate distribution to obtain $g_i$'s. So, the following section will describe our methodology to obtain $g_i$'s and to adjust the diffusion model also.  
\section{Adjusted diffusion model with Gaussian density} \label{sec_adjusteddiffusion}
In general, there is no clear way to find the exact base distribution of $X(t)$. However, we could characterize the asymptotic distribution for the multi-server queues from the literature. Many research studies on multi-server queues have shown that the limit processes of the multi-server queues are Gaussian processes, and the empirical density functions of them are also close to the Gaussian density. Listing some of those, for the time-homogeneous multi-server queues, \citet{Iglehart65} and \citet{Whitt:1982p1884} show weak convergence to the Ornstein-Uhlenbeck (OU) process, and \citet{Halfin81} proves weak convergence to Brownian motion and the OU process depending on the traffic. Therefore, for a given $t$, \emph{weak convergence provides the Gaussian distribution which is asymptotically true}. For the time-varying multi-server queues (with abandonments and retrials), as depicted in Figure \ref{fig_emp}, \citet{mandelbaum98} and \citet{Mandelbaum:2002p995} show that the empirical density is close to the Gaussian density. Furthermore, \emph{the result in \citet{Mandelbaum:2002p995} implies the limit process is a Gaussian process if the fluid model hits the critically loaded time points for a countable number of times, which is true for our model}. Therefore, for our model, it is reasonable to utilize the Gaussian distribution as a base distribution to identify $g_i$'s since the Gaussian assumption is asymptotically true.
\begin{figure}[htbp]
  \begin{center}
      \subfigure[\citet{mandelbaum98}]{
		\includegraphics[width = .45\textwidth]{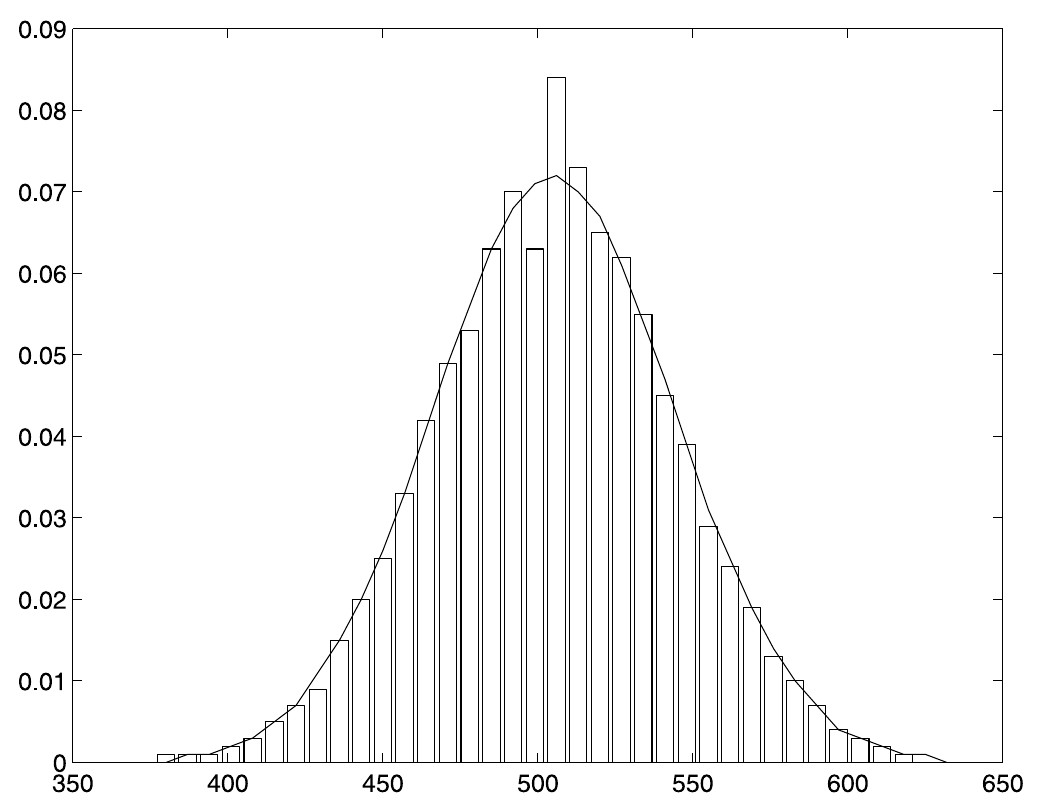}}
      \subfigure[\citet{Mandelbaum:2002p995}]{
		\includegraphics[width = .45\textwidth]{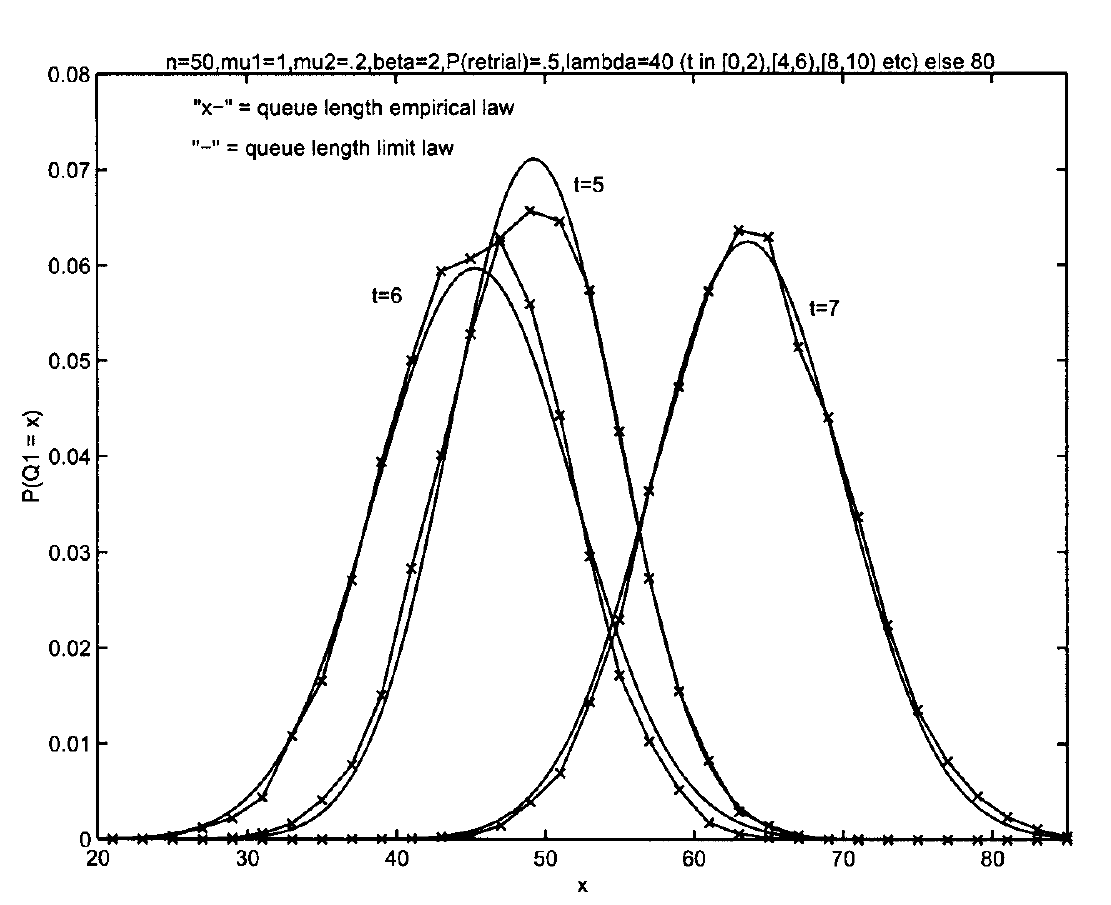}}
    \caption{Empirical density vs Gaussian density}
    \label{fig_emp}
  \end{center}
\end{figure}
Once we decide to use the Gaussian density, it provides following two additional benefits:
\begin{enumerate}
	\item The Gaussian distribution can be completely characterized by the mean and covariance matrix which can be obtained from the fluid and diffusion models.
	\item By using Gaussian density, $g_i$'s can achieve smoothness even if $f_i$'s are not smooth, which enables us to apply Theorem \ref{theo_diffusion} without additional assumptions.
\end{enumerate}
The second benefit is not obvious and hence we provide a proof of that.
\begin{lemma} \label{lem_smooth}
Let $g_i$'s be the rate functions of $Z(t)$ obtained from the Gaussian density. Then, $g_i$'s are differentiable everywhere.
\begin{proof}
Define
\begin{eqnarray}
	\phi(x,y) = \frac{1}{(2\pi)^{n/2}|\Sigma|^{1/2}}\exp \bigg(-\frac{(y-x)'\Sigma^{-1}(y-x)}{2}\bigg).\nonumber 
\end{eqnarray}
Using Gaussian density,
\begin{eqnarray*}
	g_i(t,x) = \int_{\mathbf{R}^d} f_i(t,y) \phi(x,y) dy.
\end{eqnarray*}
For $j \in \{1,\ldots, d\}$, since $\phi(x,y)$ is differentiable with respect to $x_j$ and $|f_i(t,y) \frac{d}{dx_j} \phi(x,y)|$ is integrable,
\begin{eqnarray}
	\frac{d}{dx_j}g_i(t,x) &=& \frac{d}{dx_j}\int_{\mathbf{R}^d} f_i(t,y) \phi(x,y) dy \nonumber\\
		&=& \int_{\mathbf{R}^d} f_i(t,y) \frac{d}{dx_j} \phi(x,y) dy \quad \textrm{by applying Theorem 2.27 in \citet{folland99}}, \label{eqn_folland}
\end{eqnarray}
where $x_j$ is $j^{\textrm{th}}$ component of $x$. Therefore, $g_i$ is differentiable with respect to $x_j$. \\
\end{proof}
\end{lemma}
Now, we have $g_i(\cdot,\cdot)$'s which are differentiable. Then, we can apply Theorem \ref{theo_diffusion} to obtain the diffusion model for $Z_n(t)$. Note that similar to the adjusted fluid model, once we have the distribution of $X(t)$, the adjusted diffusion model is applicable to more general cases. Therefore, we first follow the notation in Section \ref{subsec_strong} and will come back to our multi-server queues with abandonments and retrials.
\begin{proposition}[Adjusted diffusion model] \label{prop_moddiffusion} Let $g_i(\cdot,\cdot)$'s be the rate functions in $Z(t)$ obtained from Gaussian density. Define a sequence of scaled centered processes $\{V_n(t)\}$ for $t \le T$ to be
\begin{eqnarray*}
	V_n(t) = \sqrt{n}\big(Z_n(t)-\bar{Z}(t)\big), \label{eqn_030}
\end{eqnarray*}
where $Z_n(t)$ and $\bar{Z}(t)$ are solutions to equations (\ref{eqn_adjseq}) and (\ref{eqn_026}) respectively. If $f_i(t,x)$'s and $F(t, x)$ satisfy equations (\ref{eqn_024}) and (\ref{eqn_025}) respectively, then
$\lim_{n \rightarrow \infty} V_n(t) = V(t)$, where
\begin{eqnarray*}
	V(t) = \sum_{i=1}^{k} l_i \int_{0}^{t} \sqrt{g_i\big(s,\bar{Z}(s)\big)}dW_i(s) + \int_{0}^{t} \partial G\big(s,\bar{Z}(s)\big) ds, \label{eqn_031}
\end{eqnarray*}
$W_i(\cdot)$'s are independent standard Brownian motions, and $\partial G\big(t,\bar{Z}(t)\big)$ is the gradient matrix of $G\big(t,\bar{Z}(t)\big)$ with respect to $\bar{Z}(t)$. Furthermore, $V(t)$ is a Gaussian process.
\begin{proof}
From definition of $G(t,x)$ in (\ref{eqn_G}), we can easily verify that $G(t,x)$ is differentiable by Lemma \ref{lem_smooth} and hence $|G(t,x) - G(t,y)| \le M|x-y|$ implies
\begin{eqnarray*}
	\bigg|\frac{\partial}{\partial x_i} G(t,x) \bigg| \le M_i \quad \textrm{for some } M_i < \infty, t \le T, \textrm{ and } i\in \{1,\ldots, d\}.
\end{eqnarray*}
Therefore, by Theorem \ref{theo_diffusion}, we prove this proposition.
\end{proof}
\end{proposition}
\begin{corollary} \label{cor_samedistribution}
If $f_i$'s are constants or linear combinations of the components of $X(t)$. Then,
\begin{eqnarray*}
	X(t) = Z(t) \quad \textrm{in distribution}. \label{eqn_032}
\end{eqnarray*}
\begin{proof}
Using the linearity of expectation, we can verify $g_i(t,x)= f_i(t,x)$ for $i\in\{1,\ldots,k\}$.
\end{proof}
\end{corollary}
Finally, we have the adjusted fluid and diffusion models by utilizing Gaussian density. Therefore, instead of assuming measure zero at a set of non-differentiable points (as done in \citet{Mandelbaum:2002p995}), we compare the adjusted models with the empirical mean and covariance matrix. Note when we explain Theorem \ref{theo_modfluid}, we do not consider $\Sigma(t)$, the covariance matrix of $X(t)$. However, from Gaussian density, we know that $\Sigma(t)$ characterizes the base distribution and it can be obtained from Proposition \ref{prop_moddiffusion}. Therefore, we rewrite $g_i$'s to be functions of $t$, $\bar{Z}(t)$, and $\Sigma(t)$; i.e.
\begin{eqnarray}
	g_{i}\big(t, \bar{Z}(t)\big) &\rightarrow& g_{i}\big(t, \bar{Z}(t), \Sigma(t) \big) \quad \textrm{for } i\in\{1,\ldots,k\} \textrm{ and}  \label{eqn_036}\\
	G\big(t, \bar{Z}(t)\big) &\rightarrow& G\big(t, \bar{Z}(t), \Sigma(t) \big). \label{eqn_037}
\end{eqnarray}
\begin{proposition}[Mean and covariance matrix] \label{prop_modmoment}
Let $Y(t) = \bar{Z}(t) + V(t)$. Then,
\begin{eqnarray}
	E\big(Y(t)\big) &=& \bar{Z}(t) \quad \textrm{and} \label{eqn_038} \\
	Cov\big(Y(t),Y(t)\big) &=& Cov\big(V(t),V(t)\big) = \Sigma(t). \label{eqn_039}
\end{eqnarray}
The quantities $\bar{Z}(t)$ and $\Sigma(t)$ are obtained by solving the following simultaneous ordinary differential equations with initial values given by $\bar{Z}(0) = x_0$ and $\Sigma(0)=0$:
\begin{eqnarray}
	\frac{d}{dt}\bar{Z}(t) &=& \sum_{i=1}^{k} l_i g_{i}\big(t, \bar{Z}(t), \Sigma(t) \big), \label{eqn_040} \\
	\frac{d}{dt}\Sigma(t) &=& A(t) \Sigma(t) + \Sigma(t) A(t)' + B(t)B(t)', \label{eqn_041}
\end{eqnarray}
where $A(t)$ is the gradient matrix of $G\big(t,\bar{Z}(t), \Sigma(t)\big)$ with respect to $\bar{Z}(t)$, and $B(t)$ is the $d \times k$ matrix such that its $i^\textrm{th}$ column is $l_i \sqrt{g_i\big(t,\bar{Z}(t), \Sigma(t)\big)}$.
\begin{proof}
Since $V(0) = 0$, from Corollary \ref{cor_moment}, we have (\ref{eqn_038}) and (\ref{eqn_039}). By rewriting (\ref{eqn_026}) in Theorem \ref{theo_modfluid} as a differential equation form, we have (\ref{eqn_040}), and by Theorem \ref{theo_moment}, we have (\ref{eqn_041}). Note that since both $\bar{Z}(t)$ and $\Sigma(t)$ are variables, we should solve (\ref{eqn_040}) and (\ref{eqn_041}) simultaneously.
\end{proof}
\end{proposition}
Eventually, we now have the adjusted fluid and diffusion models for the general cases, and it is the time to return to our system as given in Section \ref{sec_problem}. Using Gaussian density, we can obtain the new rate functions, $g_i$'s, which correspond to $f_i$'s as follows.
\begin{eqnarray}
  g_1(t,x) &=& \lambda_t, \nonumber \\
  g_2(t,x) &=& \mu_t^2 x_2, \nonumber \\
  g_3(t,x) &=& \mu_t^1\big(n_t + (x_1-n_t)\Phi(n_t,x_1,\sigma_{1_t}) - \sigma_{1_t}^2 \phi(n_t,x_1,\sigma_{1_t})\big), \nonumber \\
  g_4(t,x) &=& \beta_t(1-p_t)\Big((x_1-n_t)\big(1-\Phi(n_t,x_1,\sigma_{1_t})\big)+\sigma_{1_t}^2\phi(n_t,x_1,\sigma_{1_t})\Big), \quad \textrm{and} \nonumber \\
  g_5(t,x) &=& \beta_tp_t\Big((x_1-n_t)\big(1-\Phi(n_t,x_1,\sigma_{1_t})\big)+\sigma_{1_t}^2\phi(n_t,x_1,\sigma_{1_t})\Big), \nonumber
\end{eqnarray}
where $\Phi(a,b,c)$ and $\phi(a,b,c)$ are function values at point $a$ of the Gaussian CDF and PDF respectively with mean $b$ and standard deviation $c$.\\
Since $f_1(t,x)$ and $f_2(t,x)$ are constant and linear with respect to $x$ respectively,  $g_1(t,x)=f_1(t,x)$ and $g_2(t,x)=f_2(t,x)$. The derivation of other $g_i(\cdot,\cdot)$'s is straightforward but requires some computational efforts and hence we provide the details in Appendix \ref{app_gi}. Note $g_3$, $g_4$, and $g_5$ include $\sigma_{1_t}$ which is currently treated as a function of $t$ but is used by the adjusted diffusion model (see equations (\ref{eqn_036}) and (\ref{eqn_037})). With the $g_i$'s above, by Proposition \ref{prop_modmoment}, we finally obtain $E[Z(t)]$ and $Cov[Z(t), Z(t)]$ for $t \le T$ and will use them to approximate the mean and covariance matrix of our original process $X(t)$ in equations (\ref{eqn_rx1}) and (\ref{eqn_rx2}).\\ 
Although we obtain the functions $g_i$'s for our adjusted models, we need some intuition regarding how $g_i$'s contribute to increasing accuracy especially in the critically loaded phases. Thus, in the next section, we revisit the inaccuracy in the previous approaches and explain how our adjusted models treat this.
\section{Discussion on function $g_i$'s} \label{sec_g}
In this section, we are going to investigate the functions $g_i$'s precisely. In order to get a clearer intuition, we consider a simple $M_t/M_t/n_t$ queue which is a special case of our original model ($\beta_t = 0$, and $\mu_t^1 = \mu_t$). Let $x(t)$ denote the number of customers in the system at time $t$. Then, $x(t)$ is the solution to the following integral equation:
\begin{eqnarray*}
x(t) = x(0) +Y_1\Big(\int_{0}^{t}\lambda_s ds\Big) + - Y_2\Big(\int_{0}^{t}\big(x(s)\wedge n_s\big)\mu_sds\Big). \label{eqn_simp_ori}
\end{eqnarray*}
Here, for convenience, define $f_1(t,x) = \lambda_t$, $f_2(t,x)= \big(x\wedge n_t\big)\mu_t$, and $F(t,x) = \lambda_t -  \big(x\wedge n_t\big)\mu_t$. 
Applying theorems in Section \ref{subsec_strong}, we have the fluid model $\bar{x}(t)$ and diffusion model $u(t)$ from the following integral equations:
\begin{eqnarray*}
\bar{x}(t) &=& x(0) +\int_{0}^{t}\lambda_s  - \big(\bar{x}(s)\wedge n_s\big)\mu_s ds, \textrm{ and} \\
u(t) &=& u(0) + \int_{0}^{t} \Big(\sqrt{\lambda_s}, \sqrt{\big(\bar{x}(s)\wedge n_s\big)\mu_s}\Big)\binom{dW_1(t)}{dW_2(t)} ds + \int_{0}^{t} \partial F(s,\bar{x}(s)) ds,
\end{eqnarray*}
where
\begin{eqnarray*}
  \partial F(t,\bar{x}(t)) = \left \{ \begin{array}{ll}
     -\mu_t & \textrm{if } \bar{x}(t) \le n_t, \\
      0 & \textrm{otherwise.}
      \end{array} \right . 
\end{eqnarray*}
Notice that the drift part $\partial F(t,\bar{x}(t))$ of the diffusion model is completely determined by the fluid model and here we might encounter a serious problem. Suppose we observe several realizations of this multi-server queue. When the $\bar{x}(t)$ is much smaller than the number of server $n_t$ (underloaded phase), then there is not great possibility that an observed process is overloaded or critically loaded. Therefore, the drift part $-\mu_t$ is valid in that sense. Now, assume that $\bar{x}(t)$ is smaller than but fairly close to $n_t$. Then, it is likely that significant fraction of the realizations could be overloaded or critically loaded. However, the drift part is still $-\mu_t$ since \emph{the possibility of being overloaded or critically loaded is completely ignored by the fluid model}. Furthermore, imagine $\bar{x}(t)$ now becomes slightly larger than $n_t$. Then, the drift part suddenly changes to zero. As a result, if $\bar{x}(t)$ is fluctuating close to $n_t$, i.e. \emph{lingering}, then the drift part of the diffusion model would repeat sudden changes between the values $-\mu_t$ and $0$. Undoubtedly, it produces sharp spikes in the diffusion model as shown in Figure \ref{fig_eg_annotated} and make the quality of the approximation worse especially near the critically loaded phase.\\
Now, we turn our attention to the functions $g_i$'s. In Section \ref{sec_adjustedfluid}, $g_i$'s in the adjusted fluid model would improve the accuracy in estimating the mean values of the system states. Then, one may ask a question how $g_i$'s affect the estimation accuracy of the covariance matrix. To answer the question, let us follow the procedure to obtain $g_2(t,\cdot)$. Note $g_1(t,\cdot) = f_1(t,\cdot)$.\\
Define $G(t,x) = g_1(t,x) - g_2(t,x) = \lambda_t - g_2(t,x)$. For a fixed $t_0$, let $x = x(t_0)$, $\mu = \mu_{t_0}$, $n=n_{t_0}$ and $ z = E[x(t_0)]$. Then, 
\begin{eqnarray}
  g_2(t_0, z) = E\big[\mu(x \wedge n)\big] = \mu\Big\{E[x\mathbb{I}_{x \le n}]+nPr[x > n]\Big\}. \label{eqn_actual}
\end{eqnarray}
From equation (\ref{eqn_actual}), we could notice the following characteristics of the function $g_2(\cdot, \cdot)$.
\begin{enumerate}
\item If $Pr[x > n] \rightarrow 1$, $g_2(t_0, z) \rightarrow \mu n$.
\item If $Pr[x > n] \rightarrow 0$, $g_2(t_0, z) \rightarrow \mu z$.
\end{enumerate}
Note that $\partial G(t,\bar{z}(t))$ changes smoothly over time between $-\mu$ and $0$ according to $Pr[x(t) > n_t]$ as $Pr[x(t) > n_t]$ changes smoothly under our Gaussian assumption (in fact, any distribution having a differentiable density works). Therefore, even if the adjusted fluid model $\bar{z}(t)$ is lingering in the vicinity of $n_t$, the drift part of the adjusted diffusion model changes smoothly over time. In the following section, we provide several experimental results and show the effectiveness of the adjusted models.
\section{Numerical results} \label{sec_numerical}
We compare our adjusted models against the fluid and diffusion models with the measure-zero assumption in \citet{Mandelbaum:2002p995} for multi-server queues with abandonments and retrials. Under the similar settings in \citet{Mandelbaum:2002p995}, we use 5,000 independent simulation runs and compare the simulation result with both methodologies. We use the constant rates for the parameters except the arrival rate. The arrival rate alternates between $45$ and $55$ every two time units. Figures \ref{fig_retrial_mean} and \ref{fig_retrial_covariance} show the estimation of mean values from one experiment. The number of servers ($n_t$) is $50$ and the service rate of each server is $1$.
\begin{figure}[htbp]
  \begin{center}
      \subfigure[ Mean numbers by assuming measure zero]{
		\includegraphics[width = .8\textwidth]{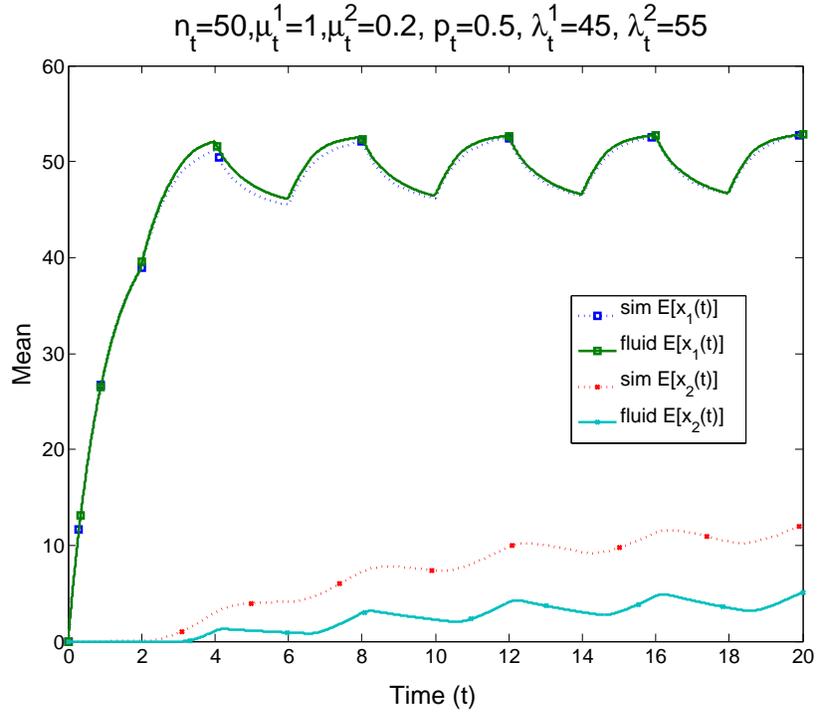}}
      \subfigure[ Mean numbers by our proposed method]{
		\includegraphics[width = .8\textwidth]{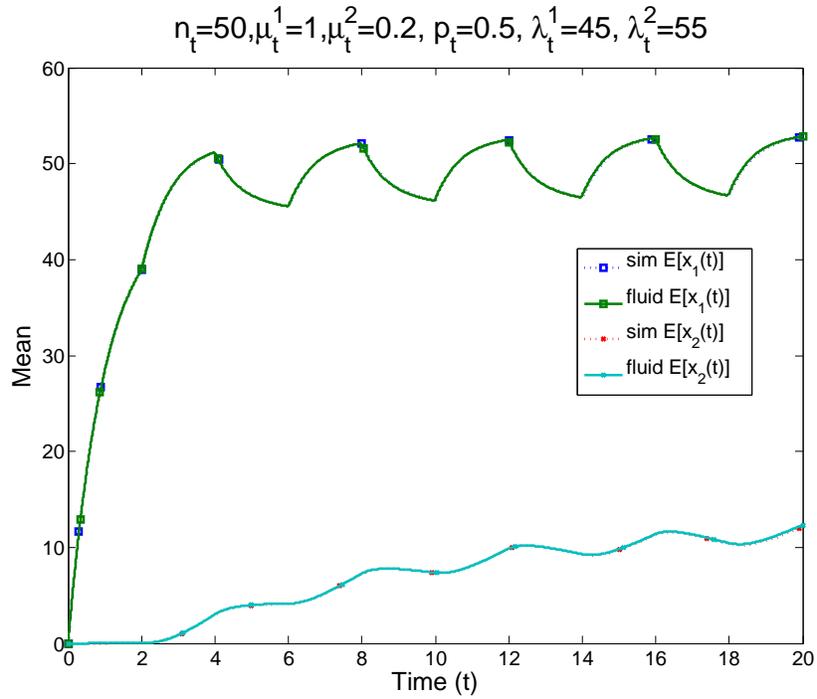}}
    \caption{Comparison of mean values, $E\big[X(t)\big]$}
    \label{fig_retrial_mean}
  \end{center}
\end{figure}
As seen in Figure \ref{fig_retrial_mean}, the number of customers in service node ($x_1(t)$) stays near the critically loaded point for a long time. As \citet{Mandelbaum:2002p995} points out, the fluid model with the measure-zero assumption shows significant estimation errors for $E\big[x_2(t)\big]$. On the other hand, our adjusted fluid model provides excellent approximation results. Especially, one can recognize remarkable improvement in the estimation of $E\big[x_2(t)\big]$. For the mean value of $x_1(t)$, our adjusted fluid model provides a lot better approximation result than the method with the measure-zero assumption.\\
\begin{figure}[htbp]
  \begin{center}
      \subfigure[ Covariance matrix by assuming measure zero]{
		\includegraphics[width = .8\textwidth]{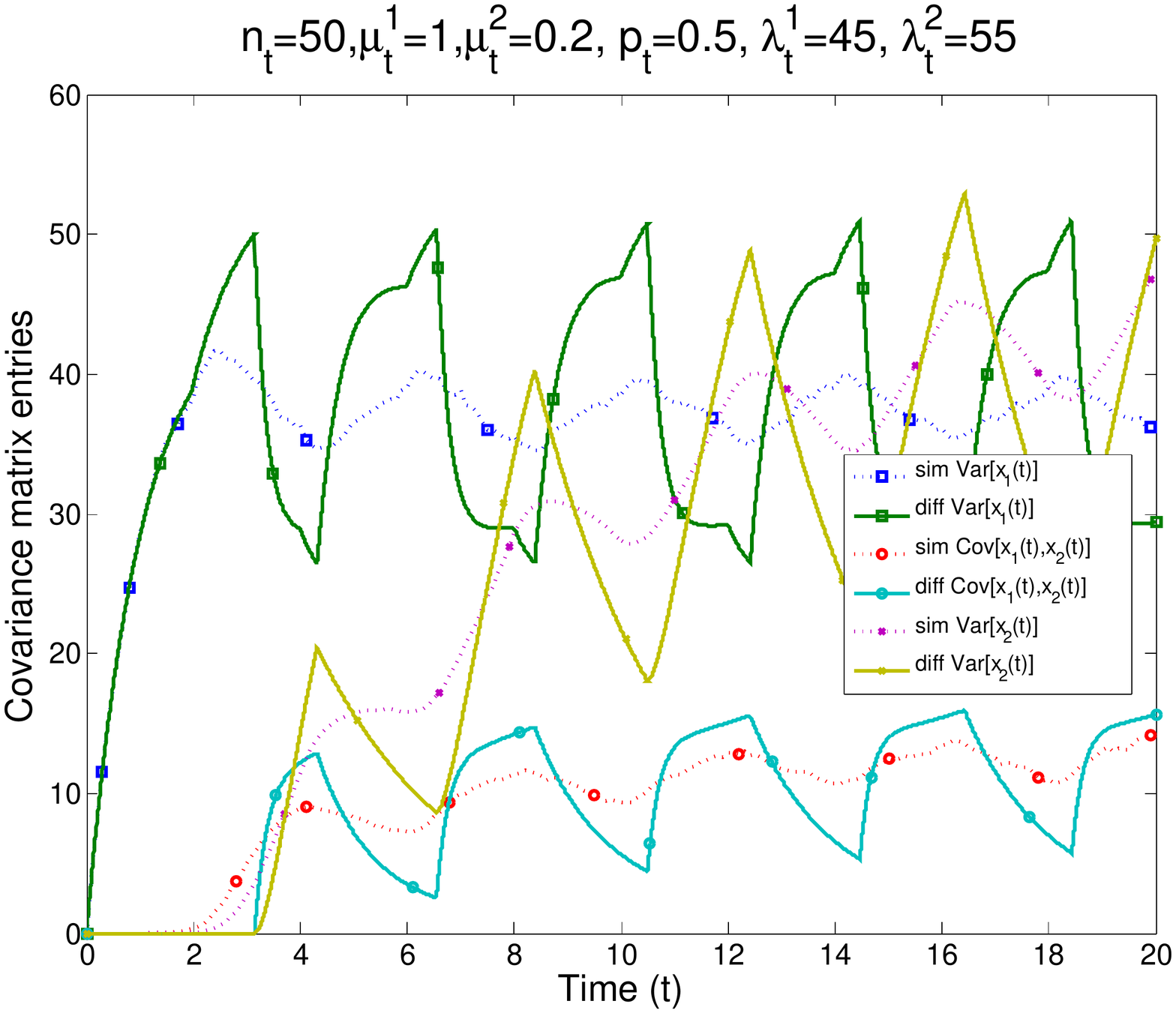}}
      \subfigure[ Covariance matrix by our proposed method]{
		\includegraphics[width = .8\textwidth]{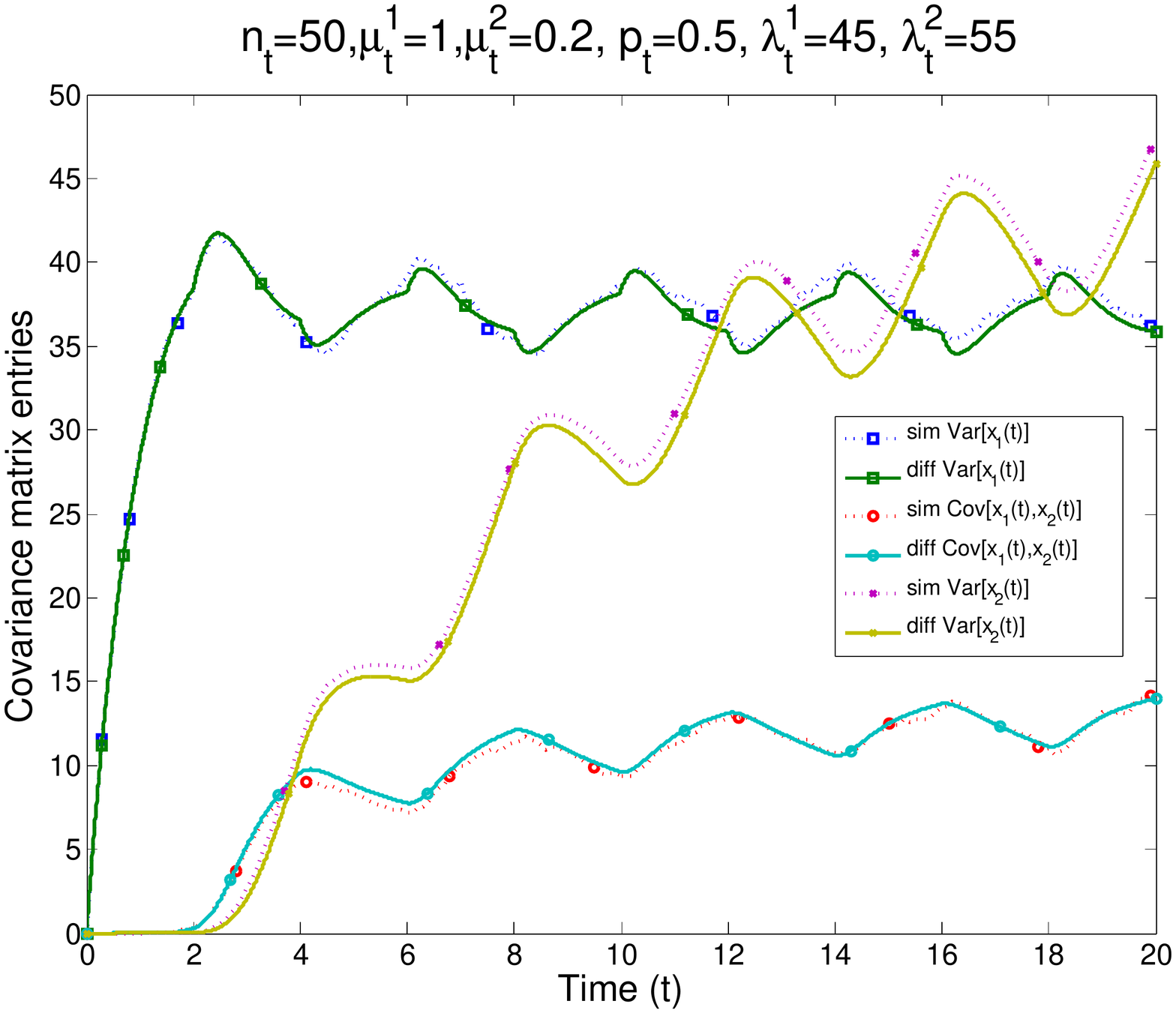}}
    \caption{Comparison of covariance matrix entries, $Cov\big[X(t),X(t)\big]$}
    \label{fig_retrial_covariance}
  \end{center}
\end{figure}
When we see the covariance matrix, we also notice our adjusted diffusion model shows dramatic improvement against the diffusion model with the measure-zero assumption.  As seen in Figure \ref{fig_retrial_covariance}, the diffusion model assuming measure zero causes ``spikes'' as also pointed out in Section \ref{subsec_inaccuracy}. Our proposed model, however, provides excellent accuracy without spikes at all.\\
Besides this specific example, in order to verify the effectiveness of our methodology, we conduct several experiments with different parameter combinations.
\begin{table}[htdp]
\caption{Experiments setting}
\begin{center}
\begin{tabular}{|c|c|c|c|c|c|c|c|c|c|c|}
\hline
exp \# & svrs & $\lambda_1$ & $\lambda_2$ & $\mu_1$ & $\mu_2$ & $\beta$ & $p$ & alter & time \\ \hline \hline
1 & 50 & 40 & 80 & 1 & 0.2 & 2.0 & 0.5 & 2 & 20 \\ \hline
2 & 50 & 40 & 60 & 1 & 0.2 & 2.0 & 0.5 & 2 & 20 \\ \hline
3 & 100 & 80 & 120 & 1 & 0.2 & 2.0 & 0.7 & 2 & 20 \\ \hline
4 & 100 & 90 & 110 & 1 & 0.2 & 2.0 & 0.7 & 2 & 20 \\ \hline
5 & 50 & 40 & 80 & 1 & 0.2 & 1.5 & 0.7 & 2 & 20 \\ \hline
6 & 50 & 40 & 60 & 1 & 0.2 & 1.5 & 0.7 & 2 & 20 \\ \hline
7 & 50 & 45 & 55 & 1 & 0.2 & 2.0 & 0.5 & 2 & 20 \\ \hline
8 & 100 & 95 & 105 & 1 & 0.2 & 2.0 & 0.5 & 2 & 20 \\ \hline
9 & 150 & 140 & 160 & 1 & 0.2 & 2.0 & 0.5 & 2 & 20 \\ \hline
10 & 150 & 100 & 190 & 1 & 0.2 & 2.0 & 0.5 & 2 & 20 \\
\hline
\end{tabular}
\end{center}
\label{tab_settings}
\end{table}
Table \ref{tab_settings} describes the setting of each experiment. In Table \ref{tab_settings}, ``svrs'' is the number of servers ($n_t$), ``alter'' is the time length for which each arrival rate lasts, and ``time'' is the end time of our analysis. We already recognize that the method assuming measure zero works well when it \emph{does not linger} too long near the non-differentiable points. For comparison, therefore, our experiments contain several cases where the system \emph{does linger} relatively long around those points as well as the cases where it does not. Experiments 1-4 are intended to see the effects of \emph{lingering} around the critically loaded points. We change $\beta_t=\beta$ and $p_t=p$ as well as the arrival rates in experiments 5-8 to see the effects of other parameters. In fact, from the other experiments not listed in Table \ref{tab_settings}, it turns out that changing other parameters does not affect estimation accuracy significantly. Experiments 9 and 10 are set to observe how larger arrival rates and number of servers affect the estimation accuracy along with the lingering effect by increasing both of them.\\
Here we explain the overall results: for the details of numerical results, see Table \ref{tab_mean_x1}-\ref{tab_var_x2} in Appendix \ref{app_table}. Similar to the results in Figures \ref{fig_retrial_mean} and \ref{fig_retrial_covariance}, we observe that lingering does debase the quality of approximations significantly when assuming measure zero. On the other hand, we see that our proposed models provide excellent accuracy for both mean and covariance matrix. Even if we increase both arrival rates and number of servers, we notice that lingering still affects the estimation accuracy significantly when assuming measure zero but it does not in our models.
\begin{figure}
\centering
      \subfigure[Average difference for all experiments]{
		\includegraphics[width = .8\textwidth]{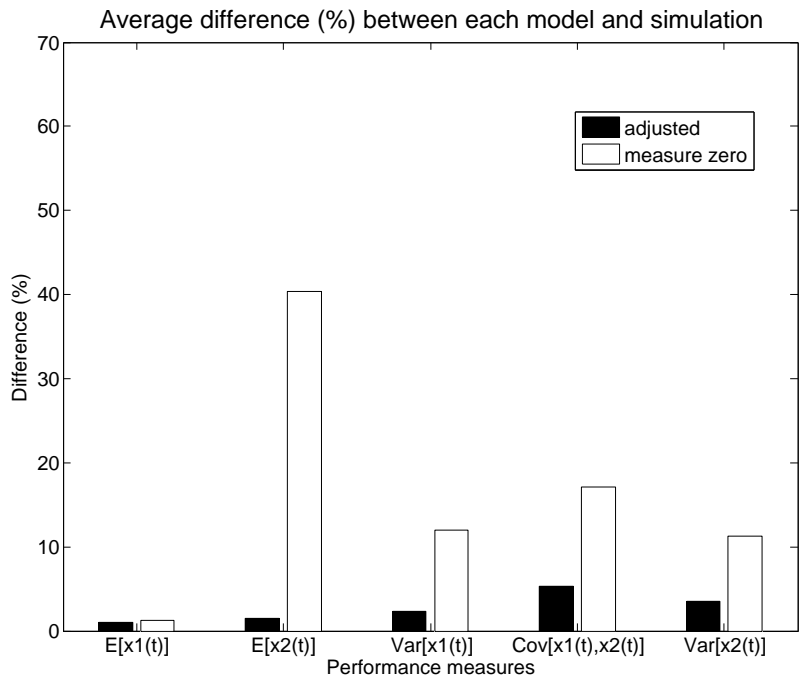}}
      \subfigure[Average difference at a critically loaded point]{
		\includegraphics[width = .8\textwidth]{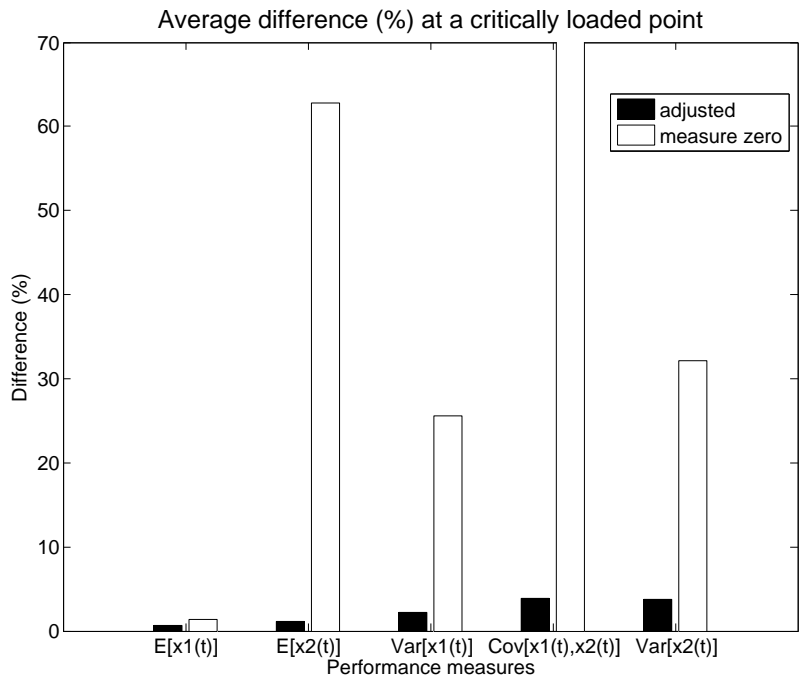}}
\caption{Average difference against simulation} \label{fig_avgresult}
\end{figure}
Figure \ref{fig_avgresult} illustrates the average percentile difference of both methods against the simulation. Figure \ref{fig_avgresult} (a) is obtained by averaging all differences in the tables (Appendix \ref{app_table} ) across time. From Figure \ref{fig_avgresult} (a), we notice that our proposed method shows promise relative to the method assuming measure zero. However, in order to clearly see the effectiveness our proposed methodology, we select the experiments 2, 4, 6, 7, 8, and 9 where lingering near the critically loaded phase occurs. We graph the differences at a critically loaded time point for those. Since the average differences are obtained from our limited experiments, it does not provide an absolute comparison between two methods. Nonetheless, we can notice that our method provides accurate estimation results consistently, but the method with measure-zero assumption results in vast inaccuracy. Note that, in Figure \ref{fig_avgresult}  (b), huge estimation difference, more than $300\%$, is observed when estimating $Cov[x_1(t),x_2(t)]$ using the method with measure-zero. However, the graph is cropped at the $70\%$ level for the illustration purpose.
\section{Conclusion} \label{sec_conclusion}
In this paper, we initially explain the strong approximations used in the analysis of multi-server queues with abandonments and retrials and show potential problems that one faces in obtaining accuracy and computational tractability especially near the critically loaded phase. The first problem stems from the fact that expectation of a function of a random vector $X$ is not equal to the value of the function of the expectation of $X$. Therefore, unless they are equal or close, the fluid model may not provide an accurate estimation of mean values of the system state. The second problem is caused by non-differentiability of rate functions which prevents applying the diffusion model in \citet{kurtz78} and causes significant estimation errors if we ignore it. Therefore, addressing these problems is quite important in order to develop accurate approximations as well as to achieve computational feasibility. For that, we proposed a methodology to obtain the exact estimation of mean values of system states and an approach to achieve computational tractability.\\
The basic idea of our approach is to construct a new stochastic process which has the fluid limit exactly same as the mean value of the system state. We proved that if rate functions in the original model satisfy the conditions to apply the fluid model, rate functions in the constructed model also satisfy those conditions. Therefore, we can apply the adjusted fluid model if we can apply the existing fluid model. It turns out that there is, in general, no computational method to obtain the adjusted fluid model exactly. Fortunately, there are several previous research studies that show the distribution of limit processes and empirical distributions are close to the Gaussian in multi-server queueing systems and hence we utilize Gaussian density to approximate it. By using Gaussian density, we see that rate functions in the constructed model are smooth and we are able to apply the diffusion model in \citet{kurtz78} even if we could not apply it to the original process.\\
To validate our proposed method, we provide several numerical examples. In the examples, we observe that our proposed method shows great accuracy compared with the fluid and diffusion approximations with measure-zero assumption (which is the only other way in the literature, to the best of our knowledge, that provides computational tractability). Due to space restriction, we have not shown all examples where our method works well. We, however, observe that in some other types of queues other than multi-server queues considered here, e.g. peer-to-peer networks, multi-class queues, the empirical density is not close to the Gaussian density. For those types of queues, one can investigate the properties of specific rate functions that affect the shape of empirical density and can devise a new methodology to find the functions $g_i(\cdot,\cdot)$'s from other density functions in the future. 
\section*{Acknowledgments}
The authors would like to thank Dr. William A. Massey and Dr. Martin I. Reiman for their inputs and valuable discussions. This research was partially supported by NSF grant CMMI-0946935.
\appendix
\appendixpage
\addappheadtotoc
\section{Proof of Lemmas \ref{lem_condition1} and \ref{lem_condition2}} \label{app_lemmas}
 \begin{proof}[Proof of Lemma \ref{lem_condition1}]
To prove this lemma, we need to show that $E\big[|X(t)|\big] \le K\Big(1+\big|E\big[X(t)\big]\big|\Big)$ for $K < \infty$ and $t \le T$. We first show it in the one-dimensional case and then extend it to the $d$-dimensional case.\\
Let, for fixed $t_0\le T$, $X = X(t_0)$ having mean $\mu$ and variance $\sigma^2$, and $f_i(X) = f_i\big(t_0,X(t_0)\big)$. Then, by Cauchy-Schwarz inequality,
\begin{eqnarray}
	E\big[|X|\big] \le \sqrt{E[X^2]} = \sqrt{\mu^2 + \sigma^2} \le  |\mu| + \sigma \le D (1 + |\mu|) \quad \textrm{for } D = \max(1,\sigma) \label{eqn_016}.
\end{eqnarray}
Now, we have the one-dimensional case and can move to the $d$-dimensional case. Suppose $X$ has a mean vector $\mu$ and a covariance matrix $\Sigma$ such that $X=(x_1, \ldots, x_d)'$, $\mu = (\mu_1, \ldots, \mu_d)'$. Then,
\begin{eqnarray}
	E\big[|X|\big] &=& E\bigg[\sqrt{\sum_{i=1}^{d}x_i^2}\bigg] \le E\bigg[\sum_{i=1}^{d}|x_i|\bigg] = \sum_{i=1}^{d} E\big[|x_i|\big] \nonumber \\
		&\le& D\bigg(d+ \sum_{i=1}^{d} |\mu_i|\bigg) \quad \textrm{by (\ref{eqn_016})} \qquad \textrm{for } D = \max(1,\sigma_1, \ldots, \sigma_d) \nonumber \\
		&\le& D\bigg(d+ d\sqrt{\sum_{i=1}^{d} \mu_{i}^{2}}\bigg) \textrm{by Cauchy-Schwarz inequality} \nonumber \\
		&=& Dd\big(1+|\mu|\big) \label{eqn_017}.
\end{eqnarray}
Now we have $E\big[|X|\big] \le K\Big(1+\big|E[X]\big|\Big)$ for the $d$-dimensional random vector $X$ where $K=Dd$. Then,
\begin{eqnarray}
	\Big|E\big[f_i(X)\big]\Big| &\le& E\Big[\big|f_i(X)\big|\Big] \le C_i + C_i E\big[|X|\big] \quad \textrm{from assumption} \nonumber \\
		&\le& C_i + C_i K \big(1 + |\mu|\big)  \nonumber \le D_i \big(1 + |\mu|\big) \quad \textrm{for } D_i = C_i + C_iK \quad \textrm{by equation (\ref{eqn_017})} \label{eqn_018}\nonumber
\end{eqnarray}
Note $g_i(t_0,\mu) =  E\big[f_i(X)\big]$. Since $|\Sigma|$ is bounded on $t\le T$, if we make $t_0>0$ arbitrary, we prove the lemma. 
\end{proof}
\begin{proof}[Proof of Lemma \ref{lem_condition2}]
For fixed $t_0 \le T$, let $X = X(t_0)$ and $Y = Y(t_0)$ and suppose $X$ and $Y$ have a same base distribution $H_0$ (we use $H$ instead of $F$ to avoid confusion with $F$ in (\ref{eqn_F})) where $E[X] = \mu_1$ and $E[Y]=\mu_2$. Then, the distribution $H_1$ of $X$ and $H_2$ of $Y$ satisfy
\begin{eqnarray*}
	H_1(x) &=& H_0(x-\mu_1), \quad \textrm{and} \\
	H_2(y) &=& H_0(y-\mu_2),
\end{eqnarray*}
respectively.
Now, we have
\begin{eqnarray*}
	\Big|E\big[F(X)\big]-E\big[F(Y)\big]\Big| &=& \bigg|\int_{\mathbf{R}^d} F(x) dH_1 - \int_{\mathbf{R}^d} F(y) dH_2\bigg|. \label{eqn_021}
\end{eqnarray*}
By transforming variables,
\begin{eqnarray}
	\Big|E\big[F(X)\big]-E\big[F(Y)\big]\Big| &=& \bigg|\int_{\mathbf{R}^d} F(x+\mu_1) dH_0 - \int_{\mathbf{R}^d} F(y+\mu_2) dH_0\bigg| \nonumber \\
		&=& \bigg|\int_{\mathbf{R}^d} \big(F(x+\mu_1) - F(x+\mu_2)\big) dH_0\bigg| \quad \textrm{by linearity}, \nonumber \\
		&\le& \int_{\mathbf{R}^d} \bigg|\big(F(x+\mu_1) - F(x+\mu_2)\big)\bigg| dH_0 \nonumber \\
		&\le& M \int_{\mathbf{R}^d} |\mu_1-\mu_2| dH_0 = M|\mu_1 - \mu_2| \quad \textrm{by assumption}. \label{eqn_022} \nonumber
\end{eqnarray}
Note $G\big(t_0, \mu_1 \big) = E\big[F(X)\big]$ and $G\big(t_0, \mu_2 \big) = E\big[F(Y)\big]$. Then, by making $t_0>0$ arbitrary, we prove the second part, i.e. if $|F(t,x)-F(t,y)|\le M|x-y|$ then $|G(t,x)-G(t,y)|\le M|x-y|$.
We can prove the first part, i.e. if $|f_i(t,x)-f_i(t,y)| \le M |x-y|$, then $|g_i(t,x) - g_i(t,y)| \le M|x-y|$, in a similar fashion and hence we have the lemma.
\end{proof}
\section{Derivation of $g_i(t,x)$'s} \label{app_gi}
For fixed $t_0>0$, let $n=n_{t_0}$, $\mu_1 = \mu_{t_0}^1$, $\beta = \beta_{t_0}$, $p=p_{t_0}$, $x_1 = x_1(t_0) \sim N(z_1, \sigma_1^2)$, and $x_2 = x_2(t_0) \sim N(z_2, \sigma_2^2)$. For $z=(z_1,z_2)'$, we have
\begin{eqnarray}
	g_3\big(t_0, z\big) &=& E\big[\mu_1(x_1 \wedge n)\big] = \mu_1\Big\{E[x_1\mathbb{I}_{x_1 \le n}]+nPr[x_1 > n]\Big\} \nonumber  \\
		&=& \mu_1\Bigg[\int_{-\infty}^{n} \frac{x}{\sqrt{2\pi}\sigma_1}\exp\bigg(-\frac{(x-z_1)^2}{2\sigma_1^2}\bigg) dx + nPr[x_1 > n]\Bigg] \nonumber \\
                &=& \mu_1\Bigg[\frac{-\sigma_1}{\sqrt{2\pi}}\int_{-\infty}^{n} -\frac{x-z_1}{\sigma_1^2}\exp\bigg(-\frac{(x-z_1)^2}{2\sigma_1^2}\bigg) dx + z_1Pr[x_1 \le n] + nPr[x_1 > n]\Bigg] \nonumber \\
                &=& \mu_1 \Bigg[-\sigma_1^2 \frac{1}{\sqrt{2\pi}\sigma_1} \exp\bigg(-\frac{(n-z_1)^2}{2\sigma_1^2}\bigg)+(z_1-n)Pr(x_1\le n) + n\Bigg]. \nonumber
\end{eqnarray}
Therefore, by making $t_0>0$ arbitrary, we have $g_3(t,x)$.\\
Note $g_4(\cdot,\cdot)$ and $g_5(\cdot,\cdot)$ are same except a constant part with respect to $x$. Therefore, it is enough to derive $g_5(\cdot,\cdot)$. We can show that
\begin{eqnarray}
	g_5\big(t_0, z\big) &=& E\big[\beta p (x_1 - n)^+\big] = \beta p \big\{E[x_1\vee n]-n\big\} \nonumber \\
        &=& \beta p \Big\{E[x_1\mathbb{I}_{x_1 > n}]+nPr[x_1 \le n] -n\Big\} \nonumber  \\
		&=& \beta p \Bigg[\int_{n}^{\infty} \frac{x}{\sqrt{2\pi}\sigma_1}\exp\bigg(-\frac{(x-z_1)^2}{2\sigma_1^2}\bigg) dx + nPr[x_1 \le n] -n\Bigg] \nonumber \\
                &=& \beta p \Bigg[\frac{-\sigma_1}{\sqrt{2\pi}}\int_{n}^{\infty} -\frac{x-z_1}{\sigma_1^2}\exp\bigg(-\frac{(x-z_1)^2}{2\sigma_1^2}\bigg) dx + z_1Pr[x_1 > n] + nPr[x_1 \le n] -n \Bigg] \nonumber \\
                &=& \beta p \Bigg[\sigma_1^2 \frac{1}{\sqrt{2\pi}\sigma_1} \exp\bigg(-\frac{(n-z_1)^2}{2\sigma_1^2}\bigg)+(z_1-n)Pr(x_1 > n) \Bigg]. \nonumber
\end{eqnarray}
Therefore, by making $t_0>0$ arbitrary, we have $g_5(t,x)$.
\section{Numerical results for Section \ref{sec_numerical}} \label{app_table}
\begin{table}[htdp]
\caption{Estimation of $E\big[x_1(t)\big]$ over time; difference from simulation}
\begin{center}
\footnotesize
\begin{tabular}{|c|c|c|c|c|c|c|c|c|c|c|c|}
\hline
\multicolumn{2}{|c|}{Experiments} & \multicolumn{10}{|c|}{Time ($t$)} \\ \hline 
\# & type & 6 & 7 & 8 & 9 & 10 & 11 & 12 & 13 & 14 & 15 \\ \hline
\multirow{2}{*}{1} & proposed & 6.52 & 0.98 & -3.39 & -1.07 & -3.05 & -0.40 & 0.91 & 0.25 & -0.69 & -0.01 \\ \cline{2-12}
& meas. 0 &4.42 & 0.82 & -3.63 & -1.94 & -3.60 & -0.23 & 0.75 & -0.15 & -2.59 & 0.11 \\ \hline
\multirow{2}{*}{2} & proposed & 2.69 & 0.44 & -3.13 & -0.82 & -1.08 & -0.32 & 0.48 & 0.15 & -0.46 & -0.05 \\ \cline{2-12}
& meas. 0 &3.35 & -0.42 & -2.92 & -1.64 & -1.18 & -1.01 & 0.85 & -0.44 & -0.36 & -0.60 \\ \hline
\multirow{2}{*}{3} & proposed & 2.33 & 0.28 & -3.11 & -1.01 & -1.36 & -0.39 & 0.10 & -0.02 & -1.68 & -0.15 \\ \cline{2-12}
& meas. 0 &2.34 & -0.42 & -2.67 & -1.55 & -1.49 & -1.00 & 0.52 & -0.49 & -0.54 & -0.53 \\ \hline
\multirow{2}{*}{4} & proposed & 1.18 & 0.14 & -1.54 & -0.30 & -0.01 & 0.12 & 0.22 & 0.22 & -0.10 & -0.02 \\ \cline{2-12}
& meas. 0 &0.65 & -0.96 & -1.98 & -1.32 & -0.94 & -0.95 & 0.04 & -0.64 & -0.61 & -0.94 \\ \hline
\multirow{2}{*}{5} & proposed & 7.04 & 1.36 & -3.67 & -0.69 & -1.38 & -0.57 & 0.80 & 0.23 & -2.82 & -0.63 \\ \cline{2-12}
& meas. 0 &5.55 & 1.04 & -3.20 & -0.93 & -1.31 & -0.53 & 0.46 & 0.06 & -1.22 & -0.18 \\ \hline
\multirow{2}{*}{6} & proposed & 3.61 & 0.76 & -3.05 & -1.13 & -0.67 & 0.18 & 1.12 & 0.20 & -0.95 & -0.25 \\ \cline{2-12}
& meas. 0 &2.53 & -0.07 & -3.01 & -1.72 & -1.46 & -0.43 & 0.60 & -0.47 & -1.57 & -0.80 \\ \hline
\multirow{2}{*}{7} & proposed & 1.93 & 0.65 & -1.06 & -0.25 & -0.63 & 0.17 & 0.12 & -0.21 & -0.65 & -0.20 \\ \cline{2-12}
& meas. 0 &0.50 & -0.86 & -2.07 & -1.51 & -1.04 & -0.73 & -0.47 & -1.07 & -0.63 & -0.76 \\ \hline
\multirow{2}{*}{8} & proposed & 0.72 & 0.07 & -0.46 & 0.04 & -0.04 & -0.14 & 0.42 & -0.07 & -0.48 & -0.01 \\ \cline{2-12}
& meas. 0 &0.04 & -0.98 & -1.40 & -0.91 & -0.57 & -0.85 & -0.13 & -0.69 & -0.73 & -0.46 \\ \hline
\multirow{2}{*}{9} & proposed & 0.81 & 0.25 & -0.96 & -0.25 & -0.11 & -0.09 & 0.38 & -0.06 & -0.24 & -0.02 \\ \cline{2-12}
& meas. 0 &0.53 & -0.50 & -1.31 & -0.88 & -0.34 & -0.61 & 0.17 & -0.51 & -0.06 & -0.32 \\ \hline
\multirow{2}{*}{10} & proposed & 6.44 & 1.18 & -4.73 & -1.73 & -2.21 & -0.45 & 0.30 & -0.01 & -1.10 & -0.11 \\ \cline{2-12}
& meas. 0 &6.46 & 0.77 & -3.83 & -1.62 & -2.84 & -0.83 & 0.84 & 0.00 & -2.77 & -0.60 \\
\hline
\end{tabular}
\end{center}
\normalsize
\label{tab_mean_x1}
\end{table}

\begin{table}[htdp]
\caption{Estimation of $E\big[x_2(t)\big]$ over time; difference from simulation}
\begin{center}
\footnotesize
\begin{tabular}{|c|c|c|c|c|c|c|c|c|c|c|c|}
\hline
\multicolumn{2}{|c|}{Experiments} & \multicolumn{10}{|c|}{Time ($t$)} \\ \hline 
\# & type & 6 & 7 & 8 & 9 & 10 & 11 & 12 & 13 & 14 & 15 \\ \hline
\multirow{2}{*}{1} & proposed & -2.00 & 3.50 & 2.36 & -0.53 & 0.57 & -1.00 & -0.99 & -0.30 & -0.44 & -0.76 \\ \cline{2-12}
& meas. 0 &11.68 & 12.60 & 7.38 & 5.88 & 11.64 & 8.18 & 5.24 & 6.29 & 10.47 & 7.82 \\ \hline
\multirow{2}{*}{2} & proposed & -2.22 & 2.71 & 1.90 & -2.44 & -0.94 & -1.82 & -0.91 & -0.10 & -0.38 & -0.76 \\ \cline{2-12}
& meas. 0 &45.00 & 53.07 & 33.49 & 37.12 & 41.73 & 44.51 & 31.55 & 37.48 & 40.57 & 43.21 \\ \hline
\multirow{2}{*}{3} & proposed & -2.49 & 1.88 & 1.00 & -3.58 & -2.09 & -3.32 & -3.08 & -3.01 & -3.15 & -4.02 \\ \cline{2-12}
& meas. 0 &28.64 & 37.65 & 19.44 & 21.88 & 24.73 & 28.38 & 16.37 & 21.45 & 22.77 & 26.96 \\ \hline
\multirow{2}{*}{4} & proposed & 0.24 & 2.66 & 1.35 & -1.68 & -0.91 & -0.19 & 0.25 & 0.45 & 0.02 & -0.53 \\ \cline{2-12}
& meas. 0 &67.95 & 69.81 & 47.03 & 51.81 & 56.69 & 59.66 & 45.16 & 50.75 & 54.48 & 57.27 \\ \hline
\multirow{2}{*}{5} & proposed & -1.01 & 4.41 & 3.16 & -0.05 & 1.55 & 0.57 & -0.58 & -0.07 & -0.09 & -1.63 \\ \cline{2-12}
& meas. 0 &9.61 & 12.28 & 7.50 & 5.73 & 11.28 & 9.42 & 5.36 & 6.00 & 9.51 & 8.02 \\ \hline
\multirow{2}{*}{6} & proposed & -2.63 & 2.48 & 2.23 & -2.39 & -1.32 & -0.84 & -0.12 & 1.04 & 0.89 & -0.04 \\ \cline{2-12}
& meas. 0 &44.23 & 51.84 & 32.45 & 35.11 & 39.27 & 43.72 & 31.00 & 35.94 & 38.83 & 41.83 \\ \hline
\multirow{2}{*}{7} & proposed & 0.33 & 3.42 & 3.00 & 1.01 & 0.70 & 0.25 & 0.41 & 0.71 & 0.27 & -0.17 \\ \cline{2-12}
& meas. 0 &78.08 & 78.96 & 60.84 & 64.86 & 69.59 & 71.19 & 59.15 & 63.20 & 67.42 & 68.95 \\ \hline
\multirow{2}{*}{8} & proposed & 2.81 & 3.03 & 2.40 & 1.45 & 1.29 & 0.58 & -0.08 & 0.41 & 0.12 & -1.11 \\ \cline{2-12}
& meas. 0 &92.68 & 90.60 & 73.97 & 77.06 & 80.98 & 81.48 & 70.82 & 74.24 & 77.96 & 78.55 \\ \hline
\multirow{2}{*}{9} & proposed & -0.86 & 1.25 & 1.44 & -0.77 & -0.19 & -0.18 & 0.08 & 0.84 & 0.42 & 0.41 \\ \cline{2-12}
& meas. 0 &80.15 & 79.90 & 57.59 & 62.03 & 67.09 & 68.91 & 55.09 & 59.98 & 64.19 & 66.14 \\ \hline
\multirow{2}{*}{10} & proposed & -2.67 & 6.62 & 3.79 & -2.50 & -0.49 & -2.18 & -1.99 & -1.35 & -1.38 & -1.21 \\ \cline{2-12}
& meas. 0 &8.53 & 23.91 & 10.73 & 8.77 & 10.78 & 13.05 & 5.67 & 8.96 & 9.13 & 11.69 \\
\hline
\end{tabular}
\end{center}
\normalsize
\label{tab_mean_x2}
\end{table}

\begin{table}[htdp]
\caption{Estimation of $Var\big[x_1(t)\big]$ over time; difference from simulation}
\begin{center}
\footnotesize
\begin{tabular}{|c|c|c|c|c|c|c|c|c|c|c|c|}
\hline
\multicolumn{2}{|c|}{Experiments} & \multicolumn{10}{|c|}{Time ($t$)} \\ \hline 
\# & method & 6 & 7 & 8 & 9 & 10 & 11 & 12 & 13 & 14 & 15 \\ \hline
\multirow{2}{*}{1} & proposed & 6.94 & 0.94 & -1.92 & -2.02 & -3.66 & -0.20 & 1.70 & -1.02 & 0.49 & 2.89 \\ \cline{2-12}
& meas. 0 &-11.03 & 2.93 & -1.93 & 17.31 & -24.44 & 1.42 & 1.66 & 14.89 & -19.73 & 4.16 \\ \hline
\multirow{2}{*}{2} & proposed & 2.84 & 3.83 & -6.05 & -0.10 & -0.50 & 4.24 & 1.62 & 2.67 & -1.02 & 1.62 \\ \cline{2-12}
& meas. 0 &-6.28 & 16.69 & 6.76 & -14.45 & -12.97 & 17.62 & 12.90 & -11.61 & -14.51 & 15.29 \\ \hline
\multirow{2}{*}{3} & proposed & 4.15 & 2.09 & -0.60 & 2.15 & -6.57 & 0.50 & -3.10 & 1.15 & 2.76 & 3.74 \\ \cline{2-12}
& meas. 0 &-0.56 & 13.38 & 7.30 & -8.74 & -12.97 & 11.92 & 4.53 & -10.16 & -2.13 & 14.22 \\ \hline
\multirow{2}{*}{4} & proposed & -0.52 & -4.36 & -2.81 & 3.07 & -0.03 & 2.96 & 0.79 & 1.27 & 3.30 & 0.35 \\ \cline{2-12}
& meas. 0 &-16.38 & 11.18 & 14.13 & -12.86 & -17.81 & 17.93 & 16.94 & -15.33 & -14.05 & 15.32 \\ \hline
\multirow{2}{*}{5} & proposed & 6.83 & -0.22 & -2.49 & 0.09 & -1.67 & -3.27 & 1.71 & -4.14 & -0.55 & 1.98 \\ \cline{2-12}
& meas. 0 &-2.30 & 1.03 & -1.69 & 10.07 & -10.59 & -1.97 & 1.43 & 5.09 & -7.69 & 3.42 \\ \hline
\multirow{2}{*}{6} & proposed & 5.22 & 0.62 & -6.25 & -0.81 & -4.32 & -1.95 & 4.41 & 1.97 & -0.61 & 4.93 \\ \cline{2-12}
& meas. 0 &-1.19 & 7.72 & 1.39 & -7.42 & -11.70 & 5.61 & 10.28 & -4.33 & -7.73 & 12.15 \\ \hline
\multirow{2}{*}{7} & proposed & 2.91 & -2.29 & -1.04 & 0.92 & 0.21 & 0.18 & 3.14 & -1.10 & 4.36 & 2.28 \\ \cline{2-12}
& meas. 0 &-17.83 & 14.52 & 18.27 & -16.55 & -22.37 & 17.07 & 20.88 & -18.77 & -18.14 & 19.01 \\ \hline
\multirow{2}{*}{8} & proposed & -1.79 & 0.65 & -0.43 & 0.83 & 3.35 & -0.71 & 3.63 & 2.10 & 1.85 & 0.72 \\ \cline{2-12}
& meas. 0 &-26.38 & 16.44 & 21.26 & -18.37 & -22.66 & 16.73 & 23.72 & -17.42 & -25.80 & 18.25 \\ \hline
\multirow{2}{*}{9} & proposed & 0.62 & -0.86 & -0.83 & 3.53 & 3.36 & 5.09 & 1.52 & 1.71 & 2.73 & -1.37 \\ \cline{2-12}
& meas. 0 &-17.84 & 13.72 & 17.09 & -14.12 & -17.40 & 19.78 & 18.07 & -16.57 & -19.03 & 14.36 \\ \hline
\multirow{2}{*}{10} & proposed & 4.48 & -0.32 & -9.84 & 1.26 & -4.24 & 3.37 & 1.32 & 1.00 & 0.22 & 1.15 \\ \cline{2-12}
& meas. 0 &4.12 & 7.27 & -6.22 & -2.69 & -5.55 & 10.87 & 3.68 & -3.27 & -2.12 & 8.98 \\
\hline
\end{tabular}
\end{center}
\normalsize
\label{tab_var_x1}
\end{table}

\begin{table}[htdp]
\caption{Estimation of $Cov\big[x_1(t),x_2(t)\big]$ over time; difference from simulation}
\begin{center}
\footnotesize
\begin{tabular}{|c|c|c|c|c|c|c|c|c|c|c|c|}
\hline
\multicolumn{2}{|c|}{Experiments} & \multicolumn{10}{|c|}{Time ($t$)} \\ \hline 
\# & type & 6 & 7 & 8 & 9 & 10 & 11 & 12 & 13 & 14 & 15 \\ \hline
\multirow{2}{*}{1} & proposed & -3.03 & -3.27 & 4.75 & 3.10 & -1.72 & -3.63 & 0.39 & -4.00 & -4.15 & 0.91 \\ \cline{2-12}
& meas. 0 &25.05 & -3.88 & 4.43 & -6.75 & 15.78 & -4.69 & -0.30 & -11.23 & 4.26 & -2.03 \\ \hline
\multirow{2}{*}{2} & proposed & -6.76 & 7.23 & 2.87 & -4.73 & 11.50 & -0.47 & 6.09 & 5.48 & 5.86 & 4.82 \\ \cline{2-12}
& meas. 0 &29.60 & -6.12 & -6.56 & 3.81 & 36.90 & -13.05 & -1.41 & 12.63 & 31.69 & -5.53 \\ \hline
\multirow{2}{*}{3} & proposed & -6.74 & -2.24 & 4.53 & -9.51 & -28.97 & -3.44 & -2.57 & -6.43 & 5.52 & -3.76 \\ \cline{2-12}
& meas. 0 &25.67 & -15.55 & 0.47 & 6.26 & 6.46 & -15.26 & -6.35 & 8.68 & 31.03 & -13.44 \\ \hline
\multirow{2}{*}{4} & proposed & -0.01 & -13.57 & -8.53 & -12.42 & -9.73 & -0.00 & -10.28 & -16.74 & -1.43 & -11.64 \\ \cline{2-12}
& meas. 0 &58.61 & -29.39 & -21.29 & 10.14 & 44.87 & -14.34 & -22.28 & 5.51 & 46.98 & -26.05 \\ \hline
\multirow{2}{*}{5} & proposed & -7.19 & -0.18 & 0.13 & 2.88 & 7.13 & -8.20 & -0.42 & 1.07 & 4.17 & -1.88 \\ \cline{2-12}
& meas. 0 &19.73 & -4.03 & -1.03 & -12.97 & 26.24 & -11.11 & -1.32 & -12.31 & 20.40 & -4.25 \\ \hline
\multirow{2}{*}{6} & proposed & 2.91 & 4.90 & 2.63 & -7.64 & -7.99 & 1.17 & 2.14 & 3.80 & 6.94 & 7.92 \\ \cline{2-12}
& meas. 0 &37.93 & -15.54 & -15.96 & -3.32 & 25.88 & -18.86 & -15.00 & 6.43 & 34.79 & -10.85 \\ \hline
\multirow{2}{*}{7} & proposed & -4.38 & -1.88 & 0.97 & -14.36 & 3.15 & -1.95 & -1.19 & -1.08 & -4.77 & -0.46 \\ \cline{2-12}
& meas. 0 &52.91 & -16.88 & -16.53 & -3.19 & 43.26 & -15.54 & -16.02 & 6.73 & 34.99 & -12.29 \\ \hline
\multirow{2}{*}{8} & proposed & -20.99 & -4.21 & -6.33 & -4.51 & 3.12 & -3.43 & -1.85 & -6.78 & -1.81 & -0.79 \\ \cline{2-12}
& meas. 0 &64.94 & -8.85 & -22.74 & 13.30 & 51.79 & -11.89 & -15.66 & 7.80 & 44.16 & -8.72 \\ \hline
\multirow{2}{*}{9} & proposed & -15.01 & -6.15 & -6.27 & 3.33 & -0.25 & 2.45 & -6.00 & -7.57 & -6.93 & -6.74 \\ \cline{2-12}
& meas. 0 &55.84 & -12.34 & -17.97 & 19.55 & 45.76 & -4.17 & -15.42 & 7.67 & 37.97 & -12.70 \\ \hline
\multirow{2}{*}{10} & proposed & -18.70 & 7.57 & -3.70 & -4.76 & 8.09 & -6.43 & -2.86 & -0.03 & 2.95 & -1.11 \\ \cline{2-12}
& meas. 0 &-21.43 & -2.63 & -5.67 & -0.67 & 8.99 & -15.71 & -4.40 & 3.66 & 4.18 & -9.87 \\
\hline
\end{tabular}
\end{center}
\normalsize
\label{tab_cov}
\end{table}

\begin{table}[htdp]
\caption{Estimation of $Var\big[x_2(t)\big]$ over time; difference from simulation}
\begin{center}
\footnotesize
\begin{tabular}{|c|c|c|c|c|c|c|c|c|c|c|c|}
\hline
\multicolumn{2}{|c|}{Experiments} & \multicolumn{10}{|c|}{Time ($t$)} \\ \hline 
\# & type & 6 & 7 & 8 & 9 & 10 & 11 & 12 & 13 & 14 & 15 \\ \hline
\multirow{2}{*}{1} & proposed & -2.15 & 3.52 & 1.48 & -0.72 & -0.34 & -0.45 & 0.78 & 1.59 & 1.31 & 0.83 \\ \cline{2-12}
& meas. 0 &6.74 & 5.31 & 2.34 & -8.01 & 5.46 & 1.00 & 1.84 & -2.78 & 2.81 & -1.20 \\ \hline
\multirow{2}{*}{2} & proposed & 1.29 & 9.81 & 8.50 & 3.31 & 7.72 & 6.70 & 6.05 & 5.84 & 5.42 & 5.91 \\ \cline{2-12}
& meas. 0 &7.06 & 14.88 & -6.56 & -0.09 & 17.07 & 12.12 & -3.90 & 5.44 & 16.94 & 13.19 \\ \hline
\multirow{2}{*}{3} & proposed & -5.60 & 2.13 & -0.71 & -5.14 & -1.93 & -3.21 & -2.71 & -1.18 & -0.61 & -0.79 \\ \cline{2-12}
& meas. 0 &-2.22 & 4.53 & -10.37 & -5.14 & 4.01 & 1.00 & -8.89 & 0.66 & 6.45 & 5.96 \\ \hline
\multirow{2}{*}{4} & proposed & 5.71 & 8.34 & 2.63 & -2.01 & -0.83 & 2.01 & -0.03 & -0.27 & 0.51 & 2.66 \\ \cline{2-12}
& meas. 0 &28.49 & 26.24 & -13.87 & -3.06 & 15.93 & 15.77 & -12.78 & 0.50 & 17.18 & 17.62 \\ \hline
\multirow{2}{*}{5} & proposed & -0.97 & 4.30 & 1.25 & -0.07 & 3.22 & 3.50 & -0.22 & -0.16 & 1.50 & 0.76 \\ \cline{2-12}
& meas. 0 &3.67 & 5.28 & 1.60 & -4.35 & 7.92 & 6.18 & 1.65 & -2.36 & 5.70 & 4.03 \\ \hline
\multirow{2}{*}{6} & proposed & 2.23 & 11.10 & 8.33 & 2.94 & 4.21 & 4.30 & 1.38 & 2.79 & 3.63 & 5.03 \\ \cline{2-12}
& meas. 0 &11.13 & 21.61 & -3.19 & -0.25 & 12.83 & 14.35 & -6.15 & 1.31 & 12.95 & 14.75 \\ \hline
\multirow{2}{*}{7} & proposed & 5.23 & 7.48 & 5.57 & 1.60 & 2.20 & 4.24 & 3.45 & 4.88 & 5.03 & 4.33 \\ \cline{2-12}
& meas. 0 &33.03 & 27.56 & -16.08 & -4.39 & 21.67 & 19.11 & -11.03 & 2.36 & 24.85 & 19.64 \\ \hline
\multirow{2}{*}{8} & proposed & 10.11 & 7.03 & 3.99 & 2.44 & 3.47 & 2.45 & 2.53 & 2.25 & 1.73 & 2.30 \\ \cline{2-12}
& meas. 0 &62.03 & 46.52 & -13.63 & 0.58 & 30.10 & 23.25 & -13.94 & -0.09 & 26.38 & 20.60 \\ \hline
\multirow{2}{*}{9} & proposed & 8.18 & 7.49 & 3.22 & 0.53 & 3.83 & 4.55 & 3.14 & 4.24 & 4.90 & 5.28 \\ \cline{2-12}
& meas. 0 &39.93 & 31.88 & -18.20 & -4.36 & 22.39 & 18.66 & -12.73 & 1.21 & 23.00 & 18.98 \\ \hline
\multirow{2}{*}{10} & proposed & -0.34 & 12.31 & 5.05 & -2.01 & 1.38 & 1.13 & -3.01 & -3.64 & -4.35 & -1.66 \\ \cline{2-12}
& meas. 0 &-5.73 & 7.15 & -1.91 & -3.68 & 0.93 & -2.52 & -8.00 & -4.34 & -3.94 & -5.72 \\
\hline
\end{tabular}
\end{center}
\normalsize
\label{tab_var_x2}
\end{table}
\newpage
\bibliographystyle{plainnat}
\bibliography{multi}
\end{document}